\documentclass[11pt]{amsart}

\usepackage[T1]{fontenc}
\usepackage[latin1]{inputenc}
\usepackage[english]{babel}
\usepackage{amsmath,amssymb,amsthm}
\usepackage{enumitem}
\usepackage{mymacros}
\usepackage[margin=1.5in]{geometry}
\usepackage{stmaryrd}
\usepackage{hyperref}

\newcommand{\mcc}{M\textsuperscript{c}Carthy}
\newcommand{\Han}{\operatorname{Han}}

\renewcommand{\MR}[1]{}

\title[Column-row property]{Every complete Pick space satisfies the column-row property}

\author[M. Hartz]{Michael Hartz}
\address{Fachrichtung Mathematik, Universit\"at des Saarlandes, 66123 Saarbr\"ucken, Germany}
\email{hartz@math.uni-sb.de}
\thanks{The author was partially supported by a GIF grant.}


\date{\today}

\begin{document}

\begin{abstract}
  In the theory of complete Pick spaces, the column-row property
  has appeared in a variety of contexts.
  We show that it is satisfied by every complete Pick space in the following strong form:
  each sequence of multipliers that induces
  a contractive column multiplication operator also induces a contractive row multiplication operator.
  In combination with known results, this yields a number of consequences.
  Firstly, we obtain multiple applications to the theory of weak product spaces,
  including factorization, multipliers and invariant subspaces.
  Secondly, there is a short proof of the characterization of interpolating
  sequences in terms of separation and Carleson measure conditions,
  independent of the solution of the Kadison--Singer problem.
  Thirdly, we find that in the theory of de Branges--Rovnyak spaces on the ball,
  the column-extreme multipliers of Jury and Martin are precisely the extreme points of the unit ball of the multiplier algebra.
\end{abstract}

\maketitle

\section{Introduction}

\subsection{Background and main result}
Given an $n$-tuple $T = (T_1,\ldots,T_n)$ of bounded linear operators on a Hilbert space $\mathcal{H}$,
one can consider the column operator
\begin{equation*}
  C =
  \begin{bmatrix}
    T_1 \\ \vdots \\ T_n
  \end{bmatrix}
  : \mathcal{H} \to \mathcal{H}^n, \quad x \mapsto
  \begin{bmatrix}
    T_1 x \\ \vdots \\ T_n x
  \end{bmatrix},
\end{equation*}
as well as the row operator
\begin{equation*}
  R =
  \begin{bmatrix}
    T_1 & \cdots & T_n
  \end{bmatrix}
  : \mathcal{H}^n \to \mathcal{H}, \quad
  \begin{bmatrix}
    x_1 \\ \vdots \\ x_n
  \end{bmatrix}
  \mapsto
  \sum_{j=1}^n T_j x_j.
\end{equation*}
In general, not much can be said about the relationship between the operator norm
of the column operator and that of the row operator. The $C^*$-identity shows that
\begin{equation*}
  \|R\| = \|R R^*\|^{1/2} = \Big\| \sum_{j=1}^n T_j T_j^* \Big\|^{1/2} \le \sqrt{n} \max_{1 \le j \le n} \|T_j\| \le \sqrt{n} \|C\|,
\end{equation*}
and similarly $\|C\| \le \sqrt{n} \|R\|$.
Easy examples using matrices with only one non-zero entry demonstrate that the factor $\sqrt{n}$ is best possible in general.

In this paper, we study the relationship between the column operator norm and the row operator norm
of tuples of multiplication operators on complete Pick spaces.
Complete Pick spaces form a class of reproducing kernel Hilbert spaces that includes the Hardy space $H^2$ on the unit disc,
the classical Dirichlet space and standard weighted Dirichlet spaces $\mathcal{D}_\alpha$ on the unit disc,
and the Sobolev space $W^2_1$ on the unit interval.
A particularly important example is the Drury--Arveson space $H^2_d$ on the unit ball $\mathbb{B}_d$ of $\mathbb{C}^d$ or of $\ell^2$,
also known as symmetric Fock space. This space is a universal complete Pick space \cite{AM00};
in addition, it plays a key role in multivariable operator theory \cite{Arveson98}.
More examples of complete Pick spaces are given by superharmonically weighted Dirichlet spaces \cite{Shimorin02}
and by certain radially weighted Besov spaces on the unit ball \cite{AHM+18a}.
We will recall the precise definition of complete Pick spaces in Subsection \ref{ss:CNP}.
Further background can be found in the book \cite{AM02}.

Given a reproducing kernel Hilbert space $\mathcal{H}$, we will write $\Mult(\mathcal{H})$ for its multiplier algebra.
Thus, $\varphi \in \Mult(\mathcal{H})$ if and only if $\varphi \cdot f \in \mathcal{H}$ whenever $f \in \mathcal{H}$.
If $\varphi \in \Mult(\mathcal{H})$, we denote the associated multiplication operator on $\mathcal{H}$ by $M_\varphi$.
The closed graph theorem implies that $M_\varphi$ is automatically bounded.

\begin{defn}
  A reproducing kernel Hilbert space $\mathcal{H}$ is said to satisfy the \emph{column-row property with constant $c \ge 1$}
  if whenever $(\varphi_n)_{n=1}^\infty$ is a sequence in $\Mult(\mathcal{H})$ with
  \begin{equation*}
    \Bigg\|
    \begin{bmatrix}
      M_{\varphi_1} \\ M_{\varphi_2} \\ \vdots
    \end{bmatrix} \Bigg\| \le 1,
  \end{equation*}
  then
  \begin{equation*}
    \|
    \begin{bmatrix}
      M_{\varphi_1} & M_{\varphi_2} & \cdots
    \end{bmatrix} \| \le c.
  \end{equation*}
\end{defn}
The key point in this definition is that the constant $c$ is independent of the length of the sequence.

In recent years, the column-row property has emerged as an important technical property
in the theory of complete Pick spaces. It plays a role in the context of corona theorems \cite{KT13,Trent04},
weak products \cite{AHM+18,CHar,JM18}, interpolating sequences \cite{AHM+17} and de Branges--Rovnyak spaces on the ball \cite{JM18b,Jury14,JM16}.
We will review some of these connections below.

If $\mathcal{H}$ is the Hardy space $H^2$ on the unit disc $\mathbb{D}$, then the norm in $\mathcal{H}$ can be expressed
as an $L^2$-norm, from which it easily follows that for any sequence of multipliers $(\varphi_n)$ on $H^2$,
\begin{equation*}
    \|
    \begin{bmatrix}
      M_{\varphi_1} & M_{\varphi_2} & \cdots
    \end{bmatrix} \|
    = \sup_{z \in \mathbb{D}} \| (\varphi_n(z)) \|_{\ell^2} = 
    \Bigg\|
    \begin{bmatrix}
      M_{\varphi_1} \\ M_{\varphi_2} \\ \vdots
    \end{bmatrix} \Bigg\|.
\end{equation*}
In particular, $H^2$ satisfies the column-row property with constant $1$.
But the behavior of $H^2$ is not typical, as the multiplier norm is typically not
even comparable to a supremum norm. The first non-trivial example
of a complete Pick space with the column-property is due to Trent \cite{Trent04},
who showed that the Dirichlet space satisfies the column-row property with constant $\sqrt{18}$. Moreover,
he showed that there are sequences of multipliers of the Dirichlet space that give bounded row multiplication operators,
but unbounded column multiplication operators. Kidane and Trent \cite{KT13} later showed that standard weighted Dirichlet spaces
$\mathcal{D}_\alpha$ satisfy the column-row property with constant $\sqrt{10}$.
In \cite{AHM+18}, Aleman, \mcc, Richter and the author showed that standard weighted Besov spaces on the unit ball $\mathbb{B}_d$ in $\mathbb{C}^d$ satisfy
the column-row property with some finite constant, possibly depending on $d$ and on the Hilbert function space.
In particular, this implies that for each $d \in \mathbb{N}$, the Drury-Arveson space $H^2_d$ satisfies the column-row property with some finite constant $c_d$, but the case $d=\infty$ remained open. In fact, the upper bound for $c_d$ obtained in \cite{AHM+18} grows exponentially in $d$.
Also in the case of $H^2_d$ for $d \ge 2$, there are sequences of multipliers that give unbounded column, but bounded row operators;
see \cite[Subsection 4.2]{AHM+18}. In fact, the basic $\sqrt{n}$-bound between column
norm and row norm is sharp in this case; see \cite[Lemma 4.8]{CHar}.
The column-row property with some finite constant was extended to radially weighted Besov spaces on the ball in \cite{AHM+18a}. We also mention
the recent paper of Pascoe \cite{Pascoe19}, where it is shown that certain spaces satisfy the column-row property
on average.
Very recently, Augat, Jury and Pascoe \cite{AJP20} showed that the column-row property fails for the full Fock space,
which can be regarded as the natural non-commutative analogue of the Drury--Arveson space.

We now state the main result of this article.

\begin{thm}
  \label{thm:main}
  Every normalized complete Pick space satisfies the column-row property with constant $1$.
\end{thm}

The normalization hypothesis is not crucial and merely assumed for convenience;
see the discussion in Subsection \ref{ss:CNP}. 
As a special case of Theorem \ref{thm:main}, we also see that the constants mentioned above for the Dirichlet space,
the standard weighted Dirichlet spaces $\mathcal{D}_\alpha$ and for the Drury--Arveson $H^2_d$ can in fact be replaced by $1$,
which is important for some applications; see for instance Theorem \ref{thm:extreme_point} below.

The proof of Theorem \ref{thm:main} occupies Section \ref{sec:proof}.
A sketch of the overall structure of the argument will be provided in Subsection \ref{ss:sketch}.
In fact, the proof yields the more general ``column-matrix property'';
see Corollary \ref{cor:column-matrix} for the precise statement.
In Theorem \ref{thm:column_matrix_pairs}, we will deduce a version of the column-row property for spaces with a complete Pick factor.

\subsection{Applications}
As alluded to above, the column-row property has appeared in a number of places in recent years.
We review some of the applications that can be obtained by combining Theorem \ref{thm:main}
with known results in the literature.

First, we consider weak product spaces, which play the
role of the classical Hardy space $H^1$ in the theory
of complete Pick spaces. These spaces go back to work of Coifman, Rochberg and Weiss \cite{CRW76}.
If $\mathcal{H}$ is a reproducing kernel Hilbert space, the weak product space is
\begin{equation*}
  \mathcal{H} \odot \mathcal{H} = \Big\{ h = \sum_{n=1}^\infty f_n g_n : \sum_{n=1}^\infty \|f_n\| \|g_n\| < \infty \Big\}.
\end{equation*}
This space is a Banach function space when equipped with the norm
\begin{equation*}
  \|h\|_{\mathcal{H} \odot \mathcal{H}} = \inf \Big\{ \sum_{n=1}^\infty \|f_n\| \|g_n\|:
  h = \sum_{n=1}^\infty f_n g_n \Big\}.
\end{equation*}
(We adopt the convention that norms without subscripts denote norms in the Hilbert space $\mathcal{H}$.)
If $\mathcal{H} = H^2$, then $\mathcal{H} \odot \mathcal{H} = H^1$, with equality
of norms. In fact, in this case, each function $h \in H^1$ can be factored as $h = f g$
with $f,g \in H^2$ and $\|h\|_{\mathcal{H} \odot \mathcal{H}} = \|f\| \|g\|$.
Jury and Martin \cite[Theorem 1.3]{JM18} showed that if $\mathcal{H}$
is a normalized complete Pick space that satisfies the column-row property with constant $c$, then
every $h \in \mathcal{H} \odot \mathcal{H}$ factors as $h = f g$, with $\|f\| \|g\| \le c \|h\|_{\mathcal{H} \odot \mathcal{H}}$. Combining their result with Theorem \ref{thm:main}, we therefore obtain the following
description of $\mathcal{H} \odot \mathcal{H}$.

\begin{thm}
  \label{thm:wp_factorization}
  Let $\mathcal{H}$ be a normalized complete Pick space and let $h \in \mathcal{H} \odot \mathcal{H}$.
  Then there exist $f,g \in \mathcal{H}$ with $h = f g$ and $\|f\| \|g\| = \|h\|_{\mathcal{H} \odot \mathcal{H}}$.
  \qed
\end{thm}

The multiplier algebra $\Mult(\mathcal{H} \odot \mathcal{H})$ of the weak product space
is the algebra of all functions that pointwise multiply $\mathcal{H} \odot \mathcal{H}$ into itself,
equipped with the norm of the multiplication operator.
In the case of $H^2$, it is easy to see that $\Mult(H^2) = H^\infty = \Mult(H^1)$,
with equality of norms.
Richter and Wick showed that if $\mathcal{H}$ is a first order Besov space on $\mathbb{B}_d$,
then $\Mult(\mathcal{H} \odot \mathcal{H}) = \Mult(\mathcal{H})$, with equivalence of norms \cite{RW16}.
For normalized complete Pick spaces $\mathcal{H}$, a description of
$\Mult(\mathcal{H} \odot \mathcal{H})$ in terms of column multiplication operators
of $\mathcal{H}$ was obtained by Clou\^atre and the author in \cite{CHar}.
As observed there, in the presence of the column-row property, this leads to the equality
$\Mult(\mathcal{H} \odot \mathcal{H}) = \Mult(\mathcal{H})$. Thus, combining \cite[Corollary 1.3]{CHar}
with Theorem \ref{thm:main}, we obtain the following result.

\begin{thm}
  \label{thm:wp_mult}
  Let $\mathcal{H}$ be a normalized complete Pick space. Then $\Mult(\mathcal{H} \odot \mathcal{H}) = \Mult(\mathcal{H})$ and
  \begin{equation*}
    \|\varphi\|_{\Mult(\mathcal{H} \odot \mathcal{H})} = \|\varphi\|_{\Mult(\mathcal{H})}
  \end{equation*}
  for all $\varphi \in \Mult(\mathcal{H})$. \qed
\end{thm}

A subspace $\mathcal{M}$ of a Banach function space $\mathcal{B}$ is said to be multiplier
invariant if $\varphi \cdot f \in \mathcal{M}$ whenever $f \in \mathcal{M}$ and $\varphi$ is a function
that multiplies $\mathcal{B}$ into itself.
Multiplier invariant subspaces of $H^2$ and of $H^1$ are described by Beurling's theorem; in particular, they are in $1$-to-$1$ correspondence. This $1$-to-$1$ correspondence
was extended to complete Pick spaces satisfying the column-row property in \cite{AHM+18}.
Combining \cite[Theorem 3.7]{AHM+18} with Theorem \ref{thm:main}, we obtain the following result.

\begin{thm}
  Let $\mathcal{H}$ be a normalized complete Pick space. Then the mappings
    \begin{align*}
      \cM &\mapsto \ol{\cM}^{\cH \odot \cH}  \\
      \cN \cap \cH &\mapsfrom \cN.
    \end{align*}
  establish a bijection between closed multiplier invariant subspaces $\mathcal{M}$ of $\mathcal{H}$
  and closed multiplier invariant subspaces $\mathcal{N}$ of $\mathcal{H} \odot \mathcal{H}$.
  \qed
\end{thm}

In Section \ref{sec:applications}, we will collect a few more applications to weak product spaces.

Let $\mathcal{H}$ be normalized complete Pick space of functions on $X$. A sequence $(z_n)$ in $X$
is said to be an \emph{interpolating sequence}  for $\Mult(\mathcal{H})$ if the evaluation map
\begin{equation*}
  \Mult(\mathcal{H}) \to \ell^\infty, \quad \varphi \mapsto (\varphi(z_n)),
\end{equation*}
is surjective. Interpolating sequences for $H^\infty$ were characterized by Carleson
in terms of what are now known as Carleson measure and separation conditions.
In \cite{AHM+17}, Carleson's theorem was extended by Aleman, \mcc, Richter and the author to normalized complete Pick spaces,
but the proof required the solution of the Kadison--Singer problem due to Marcus, Spielman
and Srivastava \cite{MSS15}. It was also observed in \cite{AHM+17}
that in the presence of the column-row property, a simpler proof that is independent
of the theorem of Marcus, Spielman and Srivastava can be given.
For the convenience of the reader, we provide the essential part of the argument in which the column-row property enters in Section \ref{sec:applications}.

Our final application of Theorem \ref{thm:main} concerns de Branges--Rovnyak spaces on the ball.
Classical de Branges--Rovnyak spaces on the unit disc play an important role
in function theory and operator theory; see the book \cite{Sarason94} for background.
Much of the classical theory was extended to the multivariable setting by Jury \cite{Jury14}
and by Jury and Martin \cite{JM18b,JM16}.
If $b$ is an element of the closed unit ball of $\Mult(H^2_d)$, the space $\mathcal{H}(b)$
is the reproducing kernel Hilbert space on $\mathbb{B}_d$ with reproducing kernel
\begin{equation*}
  \frac{1 - b(z) \overline{b(w)}}{1 - \langle z,w \rangle}.
\end{equation*}
If $d=1$, then $\Mult(H^2_d) = H^\infty$, and we recover the de Branges--Rovnyak spaces on the unit disc.
A key feature of the classical theory is a qualitative difference in the behavior of $\mathcal{H}(b)$,
depending on whether or not $b$ is an extreme point of the unit ball of $H^\infty$.
The results of \cite{JM18b,Jury14,JM16} show a similar dichotomy in the multivariable setting,
depending on whether or not $b$ is column extreme. Here, a contractive multiplier $b$ of a reproducing kernel Hilbert space $\mathcal{H}$
is said to be column extreme if there does not exist $a \in \Mult(\mathcal{H}) \setminus \{0\}$ so that
\begin{equation*}
  \Big\|
  \begin{bmatrix}
    M_b \\ M_a
  \end{bmatrix} \Big\| \le 1.
\end{equation*}
In \cite{JM16}, Jury and Martin showed that every extreme point of the closed unit ball of $\Mult(H^2_d)$ is column extreme,
and they asked if the converse holds. Combining Theorem \ref{thm:main} with a result of Jury and Martin, we obtain
a positive answer.

\begin{thm}
  \label{thm:extreme_point}
  Let $\mathcal{H}$ be a normalized complete Pick space and let $b$ belong to the closed unit ball of $\Mult(\mathcal{H})$.
  Then $b$ is an extreme point of the closed unit ball of $\Mult(\mathcal{H})$ if and only if $b$ is column extreme.
\end{thm}
This result will also be proved in Section \ref{sec:applications}.
We remark that the proof crucially uses the fact that the column-row property holds with constant $1$, as opposed
to some larger constant.

In the final section, we exhibit counterexamples to some potential strengthenings of Theorem \ref{thm:main}.
In particular, we put the main results in the context of operator spaces and
record the simple observation that the complete version of Theorem \ref{thm:main} fails,
even though some matrix versions of Theorem \ref{thm:main} hold.
Moreover, we show that Theorem \ref{thm:main} may fail in the absence of the complete Pick property.
In fact, we construct a reproducing kernel Hilbert space of holomorphic functions on the unit disc
that does not satisfy the column-row property with constant $1$ (and is not a complete Pick space).
Finally, we mention some open questions.

\subsection*{Acknowledgements}
The author is grateful to Alexandru Aleman, John \mcc\ and Stefan Richter, as well as to James Pascoe, for valuable discussions at several occasions.
Moreover, he is indebted to Mike Jury and Rob Martin for sharing Lemma \ref{lem:JM_extreme}.

\section{Preliminaries}
\label{sec:prelim}

\subsection{Kernels and multipliers}

We briefly recall some preliminaries from the theory of reproducing kernel Hilbert spaces.
For more background, the reader is referred to the books \cite{AM02,PR16}.
A reproducing kernel Hilbert space is a Hilbert space $\mathcal{H}$ of complex valued functions
on a set $X$ such that for each $w \in X$, the evaluation functional
\begin{equation*}
  \mathcal{H} \to \mathbb{C}, \quad f \mapsto f(w),
\end{equation*}
is bounded.
The \emph{reproducing kernel}  of $\mathcal{H}$ is the unique function $K: X \times X \to \mathbb{C}$
satisfying
\begin{equation*}
  \langle f, K(\cdot,w) \rangle_{\mathcal{H}}  = f(w)
\end{equation*}
for all $w \in X$ and $f \in \mathcal{H}$.
We will assume for simplicity that all reproducing kernel Hilbert spaces are separable.

A function $\varphi: X \to \mathbb{C}$ is said to be a \emph{multiplier} of $\mathcal{H}$ if
$\varphi \cdot f \in \mathcal{H}$ whenever $f \in \mathcal{H}$. We write $\Mult(\mathcal{H})$
for the algebra of all multipliers of $\mathcal{H}$.
More generally, if $\mathcal{E}$ is an auxiliary Hilbert space, we can think
of elements of $\mathcal{H} \otimes \mathcal{E}$ as $\mathcal{E}$-valued functions on $X$.
If $\mathcal{K}$ is another reproducing kernel Hilbert space on $X$ and if
$\mathcal{F}$ is another auxiliary Hilbert space, we define
$\Mult(\mathcal{H} \otimes \mathcal{E}, \mathcal{K} \otimes \mathcal{F})$
to be the space of all $B(\mathcal{E},\mathcal{F})$-valued functions on $X$
that pointwise multiply $\mathcal{H}  \otimes \mathcal{E}$ into $\mathcal{K} \otimes \mathcal{F}$.
If $\mathcal{E}$ and $\mathcal{F}$ are understood from context, we simply call elements of $\Mult(\mathcal{H} \otimes \mathcal{E}, \mathcal{H} \otimes \mathcal{F})$
multipliers of $\mathcal{H}$.
The closed graph theorem easily implies that every element of $\Mult(\mathcal{H} \otimes \mathcal{E}, \mathcal{K} \otimes \mathcal{F})$ induces a bounded multiplication operator $M_{\Phi}: \mathcal{H} \otimes \mathcal{E} \to \mathcal{K} \otimes \mathcal{F}$.
The \emph{multiplier norm} of $\Phi$ is the operator norm of $M_{\Phi}$.
In particular, we say that $\Phi$ is a \emph{contractive multiplier} if $M_\Phi$
is a contraction, i.e.\ if $\|M_\Phi\| \le 1$.
We write $\Mult_1(\mathcal{H} \otimes \mathcal{E}, \mathcal{K} \otimes \mathcal{F})$
for the closed unit ball of $\Mult(\mathcal{H} \otimes \mathcal{E}, \mathcal{K} \otimes \mathcal{F})$,
i.e.\ for the set of all contractive multipliers from $\mathcal{H} \otimes \mathcal{E}$ into $\mathcal{K} \otimes \mathcal{F}$.

An element $\Phi \in \Mult(\mathcal{H} \otimes \mathbb{C}^M, \mathcal{H} \otimes \mathbb{C}^N)$
can be identified with an $N \times M$ matrix of elements $\varphi_{i j} \in \Mult(\mathcal{H})$,
and
\begin{equation*}
  \|\Phi\|_{\Mult(\mathcal{H} \otimes \mathbb{C}^M, \mathcal{H} \otimes \mathbb{C}^N)}
  =  \bigg\|
  \begin{bmatrix}
    M_{\varphi_{1 1}} & \cdots & M_{\varphi_{1M}} \\
    \vdots & \ddots & \vdots \\
    M_{\varphi_{N1}} & \cdots & M_{\varphi_{NM}}
  \end{bmatrix} \bigg\|.
\end{equation*}
In particular, columns of $N$ multiplication operators are identified with elements
of $\Mult(\mathcal{H}, \mathcal{H} \otimes \mathbb{C}^N)$, and rows of $M$ multiplication
operators are identified with elements of $\Mult(\mathcal{H} \otimes \mathbb{C}^M, \mathcal{H})$.

A function $L: X \times X \to B(\mathcal{E})$ is said to be \emph{positive} if for all $n \in \mathbb{N}$
and all $z_1,\ldots,z_n \in X$, the operator matrix
\begin{equation}
  \label{eqn:kernel_matrix_prelim}
  \big[ L(z_i,z_j) \big]_{i,j=1}^n
\end{equation}
defines a positive operator on $\mathcal{E}^n$.
In this case, we write $L \ge 0$.
It is easy to see that in this definition, it suffices to consider distinct points $z_i$.
In particular, if $X$ itself consists of $n$ points, positivity of $L$ can be checked by considering the single operator matrix \eqref{eqn:kernel_matrix_prelim}.

We will heavily use the following well-known characterization of multipliers, which is essentially contained in
\cite[Theorem 6.28]{PR16}.

\begin{lem}
  \label{lem:multipliers_basic}
  Let $\mathcal{H}$ and $\mathcal{K}$ be reproducing kernel Hilbert spaces of functions on $X$
  with reproducing kernels $K$ and $L$, respectively.
  Let $\mathcal{E},\mathcal{F}$ be auxiliary Hilbert spaces.
  Then the following assertions are equivalent for a function $\Phi: X \to B(\mathcal{E},\mathcal{F})$:
  \begin{enumerate}[label=\normalfont{(\roman*)}]
    \item $\Phi \in \Mult_1(\mathcal{H} \otimes \mathcal{E}, \mathcal{K} \otimes \mathcal{F})$.
    \item The function
      \begin{equation*}
        X \times X \to B(\mathcal{E}), \quad (z,w) \mapsto
        K_2(z,w) I_{\mathcal{E}} - K_1(z,w) \Phi(z) \Phi(w)^*,
      \end{equation*}
      is positive.
  \end{enumerate}
\end{lem}

Moreover, we will frequently use the following basic fact, which follows directly
from the definition of positivity. For easier reference, we state it as a lemma.

\begin{lem}
  \label{lem:positive_conjugation}
  Let $\mathcal{E},\mathcal{F}$ be Hilbert spaces, let $L: X \times X \to B(\mathcal{E})$ be positive
  and let $f: X \to B(\mathcal{F},\mathcal{E})$ be a function. Then
  \begin{equation*}
    X \times X \to B(\mathcal{F}), \quad (z,w) \mapsto f(z) L(z,w) f(w)^*,
  \end{equation*}
  is positive as well. \qed
\end{lem}

\subsection{Complete Pick spaces and normalization}
\label{ss:CNP}
Complete Pick spaces are defined in terms of an interpolation property
for multipliers that recovers the classical Pick interpolation theorem in the case of the Hardy space $H^2$.
If $\mathcal{H}$ is a reproducing kernel Hilbert space of functions on $X$ with kernel $K$ and $Y \subset X$,
then $\mathcal{H} \big|_Y$ denotes the reproducing kernel Hilbert space on $Y$
with kernel $K \big|_{Y \times Y}$. The restriction map $\mathcal{H} \to \mathcal{H} \big|_Y$
is a co-isometry (see \cite[Corollary 5.8]{PR16}).
On the level of multipliers, the restriction map $\Mult(\mathcal{H}) \to \Mult(\mathcal{H} \big|_Y)$
is a complete contraction (``complete'' means that all induced maps on matrices of multipliers
also have the property in question).
The space $\mathcal{H}$ is said to be a \emph{complete Pick space} if restriction
$\Mult(\mathcal{H}) \to \Mult(\mathcal{H} \big|_Y)$ is a complete exact quotient map for all
finite sets $Y \subset X$ (``exact quotient map'' means that the closed unit ball is mapped onto the closed unit ball).
The multiplier characterization of Lemma \ref{lem:multipliers_basic} recovers the familiar formulation involving Pick matrices.
(A weak-$*$ compactness argument shows that the condition ``exact complete quotient map'' could be weakened to ``quotient map''.)

Theorems of McCullough \cite{McCullough92}, Quiggin \cite{Quiggin93} and Agler and M\textsuperscript{c}Carthy \cite{AM00} give an equivalent
characterization of the complete Pick property in terms of the reproducing kernel,
which we might also take as the definition here.
The reproducing kernel Hilbert space $\mathcal{H}$ (or its kernel $K$) is said to be \emph{irreducible}
if the underlying set $X$ cannot be partitioned into two non-empty disjoint sets $X_1,X_2$ so that $K(x_1,x_2) = 0$
for all $x_1 \in X_1, x_2 \in x_2$.
The kernel $K$ of an irreducible complete Pick space satisfies $K(z,w) \neq 0$ for all $z,w \in X$;
see \cite[Lemma 1.1]{AM00}.
By Theorem 3.1 of \cite{AM00}, the space $\mathcal{H}$ is an irreducible complete Pick space if and only if
there exist a function
$\delta: X \to \mathbb{C} \setminus \{0\}$, a number $d \in \mathbb{N} \cup \{\infty\}$ and a function
$b: X \to \mathbb{B}_d$, where $\mathbb{B}_d$ denotes the open unit ball of $\mathbb{C}^d$
if $d < \infty$ and that of $\ell^2$ if $d = \infty$, so that
\begin{equation}
  \label{eqn:CNP}
  K(z,w) = \frac{\delta(z) \overline{\delta(w)}}{1 - \langle b(z), b(w) \rangle} \quad (z,w \in X).
\end{equation}
We also refer the reader to \cite{Knese19a} for a simple and elegant proof of necessity.

A kernel $K$ is \emph{normalized} at $z_0 \in X$ if $K(z,z_0) = 1$ for all $z \in X$.
Clearly, every normalized kernel is irreducible.
By a \emph{normalized complete Pick space}, we mean an irreducible complete Pick space whose kernel
is normalized at a point.
If $K$ is the reproducing kernel of a complete Pick space that is normalized at $z_0 \in X$,
then one can achieve that in \eqref{eqn:CNP},
the function $\delta$ is the constant function $1$, hence $b(z_0) = 0$; see the discussion following
\cite[Theorem 3.1]{AM00}.

Given an irreducible complete Pick space $\mathcal{H}$ on $X$ with kernel $K$ and $z_0 \in X$, one can
consider the rescaled kernel
\begin{equation*}
  L(z,w) = \frac{K(z,w) K(z_0,z_0)}{K(z,z_0) K(z_0,w)},
\end{equation*}
which is normalized at $z_0$.
The corresponding reproducing kernel Hilbert space $\mathcal{K}$
is a normalized complete Pick space whose multipliers agree with those of $\mathcal{H}$.
More precisely,
$\Mult(\mathcal{H} \otimes \mathcal{E}, \mathcal{H} \otimes \mathcal{F}) =
\Mult(\mathcal{K} \otimes \mathcal{E}, \mathcal{K} \otimes \mathcal{F})$
for all Hilbert spaces $\mathcal{E},\mathcal{F}$, with equality of norms.
This follows, for instance, from Lemma \ref{lem:multipliers_basic} and Lemma \ref{lem:positive_conjugation}.
In particular, Theorem \ref{thm:main} for normalized complete Pick spaces
implies the same result for irreducible complete Pick spaces.
The main result was formulated for normalized spaces as this has become
the standard setting, but in the proof, it is convenient to work in the slightly more flexible
class of irreducible complete Pick spaces.
More background on rescaling kernels can be found in \cite[Section 2.6]{AM02}.

\begin{rem}
  Each general complete Pick space can be decomposed as an orthogonal direct sum
  of irreducible complete Pick spaces; see \cite[Lemma 1.1]{AM00}.
  Moreover, the direct summands are reducing for all multiplication operators,
  so if each summand satisfies the column-row property, then the entire space does.
  We omit the details as the theory of complete Pick spaces is usually only developed
  in the irreducible setting.

  We remark that sometimes, irreducibility is assumed to include the condition that $K(\cdot,w_1)$ and
  $K(\cdot,w_2)$ are linearly independent if $w_1 \neq w_2$; see \cite[Definition 7.1]{AM02}.
  However, we will not make that assumption here.
\end{rem}

\section{Proof of main result}
\label{sec:proof}

\subsection{Brief outline}
\label{ss:sketch}

We briefly discuss the main ideas that go into the proof of Theorem \ref{thm:main} (the column-row property
of complete Pick spaces).
First, we use the universality of the Drury--Arveson space and a straightforward approximation argument
to reduce Theorem \ref{thm:main} to the case where the complete Pick space $\mathcal{H}$
is the restriction of the Drury--Arveson space $H^2_d$ to a finite subset of $\mathbb{B}_d$, $d < \infty$,
and the sequence of multipliers is finite.

To treat the case when $\mathcal{H} = H^2_d \big|_F$ for a finite set $F \subset \mathbb{B}_d$,
we use a variant of the Schur algorithm. Classically, the Schur algorithm can be used
to solve Pick interpolation problems on the disc; see \cite[p.8]{AM02} for a sketch.
It consists of two steps. In the first step, one applies conformal automorphisms of the disc
to the interpolation nodes and to the targets to move one node and the corresponding target to the origin.
In the second step, one factors the desired solution $f \in H^\infty$ as $f = z g$ for another
function $g \in H^\infty$. This reduces an interpolation problem with $n$ points for $f$
to an interpolation problem with $n-1$ points for $g$, and hence makes an inductive approach possible.

Our approach to proving the column-row property of $H^2_d \big|_F$ follows a similar outline,
proceeding by induction on $|F|$. Given a contractive column multiplier $\Phi$ of $H^2_d \big|_F$
with $N$ components, we apply a conformal automorphism of $\mathbb{B}_d$ to the domain
to arrange that $0 \in F$, and a conformal automorphism of $\mathbb{B}_N$ to the range
to achieve that $\Phi(0) = 0$.
In the factorization step, we use Leech's theorem to factor
\begin{equation*}
  \Phi = \Big(
  \begin{bmatrix}
    z_1 & \cdots & z_d
\end{bmatrix} \otimes I_{N} \Big) \Psi,
\end{equation*}
where $\Psi$ is a contractive column multiplier with $d N$ components.
Taking the transpose of $\Phi$ corresponds to rearranging the column $\Psi$ into a $d \times N$ matrix.
To deal with this process, we show, using a result of Jury and Martin, that the column-row
property implies the ``column-matrix property''.
Roughly speaking, this allows us to reduce the problem for $\Phi$ on the finite set $F$
to a problem for $\Psi$ on the set $F \setminus \{0\}$ with one fewer point,
which makes it possible to apply induction.

\begin{rem}
  A version of the Schur algorithm for the Drury--Arveson space appears in \cite{ABK02},
  but the results do not seem to be directly applicable to our problem.
\end{rem}

\subsection{Reductions}
\label{ss:red}

We carry out a few straightforward reductions for the proof of Theorem \ref{thm:main}.
Firstly, it is obvious that it suffices to consider finite sequences of multipliers in the definition
of the column-row property (since the constant $c=1$ is independent of the length of the sequence).

Secondly, recall that the Drury--Arveson space $H^2_d$ is the reproducing kernel Hilbert
space on the open unit ball $\mathbb{B}_d$
of a $d$-dimensional Hilbert space, where $d \in \mathbb{N} \cup \{ \infty\}$,
with reproducing kernel
\begin{equation*}
  \frac{1}{1 - \langle z,w \rangle }.
\end{equation*}
We will use the universality of the Drury--Arveson space $H^2_d$, which is essentially the representation
\eqref{eqn:CNP}, to replace
a general irreducible complete Pick space with $H^2_d$.
Moreover, an approximation argument can be used to reduce to $d < \infty$
(once again, this works if the column-row constant of $H^2_d$ is independent of $d$).
In fact, it is sufficient to consider restrictions of $H^2_d$ to finite subsets of the ball,
which yields the reduction to finite $d$ at the same time.

If $F \subset \mathbb{B}_d$, we let $H^2_d \big|_F$ be the reproducing kernel Hilbert
space on $F$ whose reproducing kernel is the restriction of the kernel of $H^2_d$ to $F \times F$.

\begin{lem}
  \label{lem:reduction_finite}
  In order to prove Theorem \ref{thm:main}, it suffices to show that for each $d \in \mathbb{N}$ and all
  finite sets $F \subset \mathbb{B}_d$, the space $H^2_d \big|_F$ satisfies the column-row property
  with constant $1$.
\end{lem}

\begin{proof}
  Suppose that $H^2_d \big|_F$ satisfies the column-row property with constant $1$
  for all $d \in \mathbb{N}$ and all finite sets $F \subset \mathbb{B}_d$.
  Let $\mathcal{H}$ be a normalized complete Pick space on $X$.
  Thus, the reproducing kernel $K$ of $\mathcal{H}$ is of the form
  \begin{equation*}
    K(z,w) = \frac{1}{1 - \langle b(z), b(w) \rangle }
  \end{equation*}
  for a function $b: X \to \mathbb{B}_{e}$ and a number $e \in \mathbb{N} \cup \{\infty\}$.

  Let $N \in \mathbb{N}$ and let $\Phi \in \Mult_1(\mathcal{H},\mathcal{H} \otimes \mathbb{C}^N)$ be a contractive
  column multiplier of $\mathcal{H}$.
  Our goal is to show that the transposed function $\Phi^T$ is a contractive row multiplier of $\mathcal{H}$, i.e. that \ $\Phi^T \in \Mult_1(\mathcal{H} \otimes \mathbb{C}^N, \mathcal{H})$.
  By Lemma \ref{lem:multipliers_basic}, this is equivalent to showing that
  for any finite collection of points $x_1,\ldots,x_n \in X$,
  the matrix
  \begin{equation}
    \label{eqn:cr_finite_red_goal}
    \big[ K(x_i,x_j) ( 1 - \Phi^T(x_i) (\Phi^T(x_j))^*) \big]_{i,j=1}^n
  \end{equation}
  is positive.

  Since $\Phi \in \Mult_1(\mathcal{H} , \mathcal{H} \otimes \mathbb{C}^N)$, we know
  from Lemma \ref{lem:multipliers_basic} that the matrix
  \begin{equation}
    \label{eqn:cr_finite_red}
    \big[
    K(x_i,x_j) ( I_N - \Phi(x_i) \Phi(x_j)^*) \big]_{i,j=1}^n
  \end{equation}
  is positive.
  Consider the points $b(x_1), \ldots, b(x_n) \in \mathbb{B}_e$,
  which are contained in a subspace of dimension $d \le n$.
  Thus, there exist points $\lambda_1,\ldots,\lambda_n \in \mathbb{B}_d$ so
  that
    $\langle \lambda_i, \lambda_j \rangle = \langle b(x_i), b(x_j) \rangle$
  for $1 \le i, j \le n$ and so that $\lambda_i = \lambda_j$ if and only if $b(x_i) = b(x_j)$.
  Hence,
  \begin{equation*}
    K(x_i,x_j) = \frac{1}{1 - \langle \lambda_i, \lambda_j \rangle }
  \end{equation*}
  for $1 \le i,j \le n$. Let $F= \{\lambda_1,\ldots,\lambda_n\} \subset \mathbb{B}_d$.
  We wish to define
  \begin{equation*}
    \Psi: F \to B(\mathbb{C}, \mathbb{C}^N) \text{ by } \Psi(\lambda_j) = \Phi(x_j) \text{ for } 1 \le j \le n.
  \end{equation*}
  To see that $\Psi$ is well defined, we have to check that $\Phi(x_j) = \Phi(x_i)$ if $b(x_i) = b(x_j)$.
  But if $b(x_i) = b(x_j)$, then $K(\cdot,x_i) = K(\cdot,x_j)$,
  hence $f(x_i) = f(x_j)$ for all $f \in \mathcal{H} \otimes \mathbb{C}^N$ and therefore
  $\Phi(x_i) = \Phi(x_j)$ as $1 \in \mathcal{H}$ by the normalization assumption.
  Thus, $\Psi$ is well defined.

  Having defined $\lambda_1,\ldots,\lambda_n$ and $\Psi$, we can now rewrite \eqref{eqn:cr_finite_red} as
  \begin{equation*}
    \Big[ \frac{1}{1 - \langle \lambda_i, \lambda_j \rangle } ( I_N - \Psi(\lambda_i) \Psi(\lambda_j)^*) \Big]_{i,j=1}^n \ge 0.
  \end{equation*}
  By Lemma \ref{lem:multipliers_basic}, this means that $\Psi \in \Mult_1(H^2_d \big|_F, H^2_d \big|_F \otimes \mathbb{C}^N)$.
  Since $H^2_d \big|_F$ satisfies the column-row property with constant $1$ by assumption,
  it follows that $\Psi^T \in \Mult_1(H^2_d \big|_F \otimes \mathbb{C}^N, H^2_d \big|_F)$, so that
  by the same lemma,
  \begin{equation*}
    \Big[
    \frac{1}{1 - \langle \lambda_i, \lambda_j \rangle } ( 1 - \Psi^T(\lambda_i) (\Psi^T(\lambda_j))^*) \Big]_{i,j=1}^n \ge 0.
  \end{equation*}
  Recalling the choice of $\lambda_j$ and $\Psi$, we see that the matrix in \eqref{eqn:cr_finite_red_goal} is
  positive, which finishes the proof.
\end{proof}

The proof above shows that we could assume that $|F| = d$,
but it is convenient to keep the size of $F$ and the dimension $d$ independent.

\subsection{Conformal automorphisms on the domain}

For the remainder of this section, we assume that $d \in \mathbb{N}$.
Our next task is to understand the action of biholomorphic
automorphisms of $\mathbb{B}_d$ on the domain of multipliers, which
is needed in the first step of the Schur algorithm.
Let $\Aut(\bB_d)$ denote the group of biholomorphic automorphisms of $\bB_d$;
see \cite[Section 2.2]{Rudin08} for background.
It is well known that $\Aut(\bB_d)$ acts transitively on $\bB_d$,
see \cite[Theorem 2.2.3]{Rudin08}. We require the following basic facts
about $\Aut(\bB_d)$, which can be found in \cite[Therem 2.2.5]{Rudin08}.

\begin{lem}
  \label{lem:auto_basic}
  Let $\theta \in \Aut(\bB_d)$ and let $a = \theta^{-1}(0)$.
  \begin{enumerate}[label=\normalfont{(\alph*)}]
    \item $\theta$ extends to a homeomorphism of $\ol{\bB_d}$.
    \item The identity
      \begin{equation*}
        1 - \langle \theta(z), \theta(w) \rangle = \frac{( 1- \langle a,a \rangle)( 1 - \langle z, w \rangle)}{(1 - \langle z, a \rangle)(1 - \langle a, w \rangle)}
      \end{equation*}
      holds for all $z,w \in \ol{\bB_d}$.
  \end{enumerate}
\end{lem}

It is well known that every conformal automorphism of $\bB_d$ induces a
completely isometric composition operator on $\Mult(H^2_d)$. This can be deduced from part (b) of Lemma \ref{lem:auto_basic}. We require the following variant of this fact, which also easily follows from the same identity.

\begin{lem}
  \label{lem:conformal_domain}
  Let $F \subset \bB_d$, let $\cE,\cF$ be Hilbert spaces and let $\Phi: F \to B(\mathcal{E},\mathcal{F})$
  be a function.
  Let $\theta \in \Aut(\mathbb{B}_d)$.
  Then $\Phi$ is a contractive multiplier of $H^2_d \big|_F$ if and only if $\Phi \circ \theta$
  is a contractive multiplier of $H^2_d \big|_{\theta^{-1}(F)}$.
\end{lem}

\begin{proof}
  Let $K(z,w) = \frac{1}{1 - \langle z,w \rangle}$ be the reproducing kernel of $H^2_d$.
  Suppose that $\Phi$ is a contractive multiplier of $H^2_d \big|_F$. Then Lemma \ref{lem:multipliers_basic}
  shows that
  \begin{equation*}
    (z,w) \mapsto  K(z,w)(I_{\mathcal{F}} - \Phi(z) \Phi(w)^*) \ge 0
  \end{equation*}
  on $F \times F$, hence
  \begin{equation}
    \label{eqn:conf_pos}
    (z,w) \mapsto K(\theta(z),\theta(w))(I_{\mathcal{F}} - \Phi(\theta(z)) \Phi(\theta(w))^*) \ge 0
  \end{equation}
  on $\theta^{-1}(F) \times \theta^{-1}(F)$.
  Let $a = \theta^{-1}(0)$.
  Part (b) of Lemma \ref{lem:auto_basic} implies that
  \begin{equation*}
    K(\theta(z),\theta(w)) = \frac{K(z,w) K(a,a)}{K(z,a) K(a,w)}
  \end{equation*}
  for all $z,w \in \mathbb{B}_d$, hence
  \begin{align*}
    &K(z,w)(I_{\mathcal{F}} - \Phi(\theta(z)) \Phi(\theta(w))^*) \\
    = \, &\frac{1}{K(a,a)} K(z,a) K(\theta(z),\theta(w)) (I_{\mathcal{F}} - \Phi(\theta(z)) \Phi(\theta(w))^*) K(a,w),
  \end{align*}
  which is positive as a function of $(z,w)$ on $\theta^{-1}(F) \times \theta^{-1}(F)$ by
  Lemma \ref{lem:positive_conjugation} and \eqref{eqn:conf_pos}.
  This means that $\Phi \circ \theta$ is a contractive multiplier
  of $H^2_d \big|_{\theta^{-1}(F)}$ by Lemma \ref{lem:multipliers_basic}.
  The converse follows by consideration of $\theta^{-1}$.
\end{proof}

\subsection{Conformal automorphisms on the range}

Next, we study the action of conformal automorphism on the range of multipliers.
Observe that conformal automorphisms of the unit ball act on both row vectors and column vectors.
More precisely, for $N \in \mathbb{N}$, the group $\Aut(\bB_N)$ acts on the closed unit ball of $M_{1,N}(\bC)$ and on the closed
unit ball of $M_{N,1}(\bC)$.
The next lemma shows that this gives an action on contractive row and column multipliers.

\begin{lem}
  \label{lem:conformal_range}
  Let $\mathcal{H}$ be a reproducing kernel Hilbert space of functions on a set $F$ and
  let $\theta \in \Aut(\mathbb{B}_N)$.
  \begin{enumerate}[label=\normalfont{(\alph*)}]
    \item
  Let $\Phi: F \to M_{1,N}$ be a function with $\|\Phi(z)\| \le 1$
  for all $z \in F$. Then $\Phi$ is a contractive
  multiplier of $\mathcal{H}$ if and only if $\theta \circ \Phi$ is a contractive
  multiplier of $\mathcal{H}$.
    \item
  Let $\Phi: F \to M_{N,1}$ be a function with $\|\Phi(z)\| \le 1$
  for all $z \in F$. Then $\Phi$ is a contractive
  multiplier of $\mathcal{H}$ if and only if $\theta \circ \Phi$ is a contractive
  multiplier of $\mathcal{H}$.
  \end{enumerate}
\end{lem}

We will provide two proofs of Lemma \ref{lem:conformal_range}. The first proof is elementary, but requires
some computations. The second proof is shorter, but uses dilation theory.

In the first proof, we require the following analogue of the formula in part (b) of Lemma \ref{lem:auto_basic}.
Formulas of this type are certainly known, see for instance \cite[Chapter 8]{IS85}.
Since we do not have a reference for the precise formula we need, we provide the proof.

\begin{lem}
  \label{lem:automorphism_columns}
  Let $\theta \in \Aut(\mathbb{B}_N)$ and let $a = \theta(0) \in M_{N,1}(\mathbb{C})$. Then for all
  $z,w$ in the closed unit ball of $M_{N,1}$, the identity
  \begin{equation*}
    I - \theta(z) \theta(w)^* = (I - a a^*)^{1/2} (I - z a^*)^{-1} (I - z w^*)
    (I - a w^*)^{-1} (I - a a^*)^{1/2}
  \end{equation*}
  holds.
\end{lem}

\begin{proof}
  For $a \in \mathbb{B}_N$ and $z \in \overline{\mathbb{B}_N}$, let
  \begin{equation*}
    \theta_a(z) = \frac{a - P_a z - s_a Q_a z}{1 - \langle z,a \rangle},
  \end{equation*}
  where $P_a$ is the orthogonal projection onto $\mathbb{C} a$, $Q_a = I - P_a$ and $s_a = (1 - \|a\|^2)^{1/2}$.
  Then $\theta_a \in \Aut(\mathbb{B}_N)$ is an involution that takes $0$ to $a$ and vice versa, see \cite[Theorem 2.2.2]{Rudin08}.
  Moreover, \cite[Theorem 2.2.5]{Rudin08} shows that the automorphism $\theta^{-1}$ is of the form
  \begin{equation*}
    \theta^{-1} = U \circ \theta_a,
  \end{equation*}
  where $a = \theta(0)$ and $U$ is unitary. Thus, $\theta = \theta_a \circ U^*$, and it suffices to show the lemma
  for $\theta = \theta_a$.

  We claim that the alternate formula
  \begin{equation}
    \label{eqn:automorphism_column}
    \theta_a(z) = (I - a a^*)^{1/2} (I - z a^*)^{-1} (a - z) (1 - a^* a)^{-1/2}
  \end{equation}
  holds.
  To see this, we will compare inner products with $a$ and with vectors in $(\mathbb{C} a)^\bot$. 
  Let $\tau_a(z)$ denote the right-hand side of \eqref{eqn:automorphism_column}.
  Using the basic identities $(I - a a^*)^{1/2} a = a (1 - a^* a)^{1/2}$
  and $(I - a z^*)^{-1} a = a (1 - z^* a)^{-1}$,
  we find that
  \begin{equation*}
    \langle \tau_a(z), a \rangle
    = \langle (I - z a^*)^{-1} (a - z), a \rangle  = \frac{\langle a-z,a \rangle}{1 - \langle z,a \rangle }
    = \langle \theta_a(z),a \rangle.
  \end{equation*}
  On the other hand, let $v \in (\mathbb{C} a)^{\bot}$. Observe that
  \begin{equation*}
    (I - z a^*)^{-1} (a-z)
    = a - z (1 - a^* z)^{-1} ( 1- a^* a),
  \end{equation*}
  which is easily checked by multiplying both sides with $I-z a^*$ on the left.
  Moreover, $(I - a a^*)^{1/2} v = v$, for instance by expanding $(1 - a a^*)^{1/2}$ in a power series.
  Thus,
  \begin{align*}
    \langle \tau_a(z),v \rangle &= \langle (I - z a^*)^{-1} (a - z) (1 - a^* a)^{-1/2} ,v \rangle  \\
    &= - \langle z (1 - a^* z)^{-1}(1 - a^* a)^{1/2}, v \rangle 
    = \langle \theta_a(z),v \rangle.
  \end{align*}
  This proves \eqref{eqn:automorphism_column}.

  Using \eqref{eqn:automorphism_column}, we now conclude that
  \begin{align*}
    I - \theta_a(z) \theta_a(w)^*
    = &(I - a a^*)^{1/2} (I - z a^*)^{-1}
    \Big[ (I - z a^*) (I - a a^*)^{-1} (I - a w^*) \\ &\quad - (a - z)(1 - a^* a)^{-1} (a^* - w^*) \Big]
    (I - a w^*)^{-1} (I - a a^*)^{1/2}.
  \end{align*}
  To finish the proof, one checks that the quantity in square brackets
  equals $I - z w^*$, which is a straightforward computation
  by repeatedly using the identity $(I - a a^*)^{-1} a = a (1 - a^* a)^{-1}$.
\end{proof}

We are now ready for the first proof of Lemma \ref{lem:conformal_range}.

\begin{proof}[First proof of Lemma \ref{lem:conformal_range}]
  Let $K$ be the reproducing kernel of $\mathcal{H}$.

  (a) Suppose that $\Phi$ is a contractive row multiplier of $\mathcal{H}$.
  Then
  \begin{equation*}
    K(z,w) (1  - \Phi(z) \Phi(w)^*) \ge 0
  \end{equation*}
  as a function of $(z,w)$ by Lemma \ref{lem:multipliers_basic}.
  Let $a = \theta^{-1}(0) \in M_{1,N}$.
  Notice that if $x,y \in M_{1,N}$, then $x y^* = \langle x^T, y^T \rangle$. Thinking of $\Aut(\mathbb{B}_N)$ as acting on rows
  and applying part (b) of Lemma \ref{lem:auto_basic},
  we find that
  \begin{equation*}
    K(z,w) (1 - \theta(\Phi(z)) \theta(\Phi(w))^*)
    = K(z,w) \frac{(1 - a a^*) (1 - \Phi(z) \Phi(w)^*)}{(1 - \Phi(z) a^*)(1 - a \Phi(w)^*)},
  \end{equation*}
  which is positive as a function of $(z,w)$ by Lemma \ref{lem:positive_conjugation}.
  Hence $\theta \circ \Phi$ is a contractive multiplier as well.
  The converse follows by considering $\theta^{-1}$.

  (b) We use similar reasoning as in (a), but applying Lemma \ref{lem:automorphism_columns} in place of Lemma \ref{lem:auto_basic}.
  Explicitly, suppose that $\Phi$ is a contractive column multiplier of $\mathcal{H}$. Then
  \begin{equation*}
    K(z,w) (I - \Phi(z) \Phi(w)^*) \ge 0.
  \end{equation*}
  Let $a = \theta(0)$. We now think of $\Aut(\mathbb{B}_N)$ as acting  on columns.
  By Lemma \ref{lem:automorphism_columns}, we have
  \begin{gather*}
    K(z,w) ( I - \theta(\Phi(z)) \theta(\Phi(w))^*) \\
    = (I - a a^*)^{1/2} (I - \Phi(z) a^*)^{-1} K(z,w)(I - \Phi(z) \Phi(w)^*) (I - a \Phi(w)^*)^{-1} (I - a a^*)^{1/2},
  \end{gather*}
  which is positive by Lemma \ref{lem:positive_conjugation}.
  Hence $\theta \circ \Phi$ is a contractive multiplier of $\mathcal{H}$. The converse again follows
  by considering $\theta^{-1}$.
\end{proof}

We now provide a second proof of Lemma \ref{lem:conformal_range}, which is dilation theoretic.
Recall that a tuple $T = (T_1,\ldots,T_N)$ of operators on a Hilbert space is said to be a row
contraction if the row operator
\begin{equation*}
  \begin{bmatrix}
    T_1 & \ldots & T_N
  \end{bmatrix}
\end{equation*}
has operator norm at most $1$, and a column contraction if the column operator
\begin{equation*}
  \begin{bmatrix}
    T_1 \\ \vdots \\ T_N
  \end{bmatrix}
\end{equation*}
has operator norm at most $1$. If a tuple $T = (T_1,\ldots,T_N)$ of commuting operators is either a row or a column contraction
and if $\theta \in \Aut(\mathbb{B}_N)$ with component functions $\theta_1,\ldots,\theta_N$,
then the explicit formula for elements  of $\Aut(\mathbb{B}_N)$ given in \cite[Theorem 2.2.5]{Rudin08} shows that one can
define an operator tuple $\theta(T) = (\theta_1(T), \ldots, \theta_N(T))$ by means of a norm convergent power series.
(Using more machinery, one could also use the Taylor functional calculus to define $\theta(T)$.)
Moreover, if $T$ is in fact a tuple of multiplication operators on a reproducing kernel Hilbert space, say $T_{j} = M_{\varphi_j}$,
and $\Phi = (\varphi_1, \ldots, \varphi_N)$, then $\theta_j(T) = M_{\theta_j \circ \Phi}$. Therefore,
Lemma \ref{lem:conformal_range} is also a special case of the following operator theoretic result.

\begin{prop}
  Let $T = (T_1,\ldots,T_N)$ be a tuple of commuting operators on a Hilbert space and let $\theta \in Aut(\mathbb{B}_N)$.
  \begin{enumerate}[label=\normalfont{(\alph*)}]
    \item If $T$ is a row contraction, then the tuple $\theta(T)$ is also a row contraction.
    \item If $T$ is a column contraction, then the tuple $\theta(T)$ is also a column contraction.
  \end{enumerate}
\end{prop}

\begin{proof}
  (a) Since $T$ is a row contraction, the Drury--M\"uller--Vasilecu--Arveson dilation theorem
  (see \cite[Theorem 8.1]{Arveson98}) shows that
  \begin{equation*}
    \big\|
    \begin{bmatrix}
      \theta_1(T) & \ldots & \theta_N(T)
    \end{bmatrix} \big\|
    \le \big\|
    \begin{bmatrix}
      \theta_1 & \ldots & \theta_N
    \end{bmatrix}
    \big\|_{\Mult(H^2_N \otimes \mathbb{C}^N, H^2_N)}.
  \end{equation*}
  On the other hand, it is well known (and directly follows from Lemma \ref{lem:multipliers_basic}) that
  \begin{equation*}
    \big\|
    \begin{bmatrix}
      z_1 & \ldots & z_N
    \end{bmatrix}
    \big\|_{\Mult(H^2_N \otimes \mathbb{C}^N, H^2_N)} \le 1,
  \end{equation*}
  hence Lemma \ref{lem:conformal_domain}, applied with $F = \mathbb{B}_d$, 
  implies that
  \begin{equation*}
    \big\|
    \begin{bmatrix}
      \theta_1 & \ldots & \theta_N
    \end{bmatrix}
    \big\|_{\Mult(H^2_N \otimes \mathbb{C}^N, H^2_N)} \le 1.
  \end{equation*}
  Therefore, $\theta(T)$ is a row contraction.

  (b)
  Consider the adjoint tuple $T^* = (T_1^*,\ldots,T_N^*)$, which is a row contraction.
  Define $\widetilde \theta(z) = \overline{\theta(\overline{z})}$ for $z \in \overline{\mathbb{B}_N}$, where the complex
  conjugations are defined componentwise. Clearly, $\widetilde \theta \in \Aut(\mathbb{B}_N)$. Thus, part (a)
  shows that $\widetilde \theta(T^*)$ is a row contraction. Moreover, $\widetilde \theta(T^*) = \theta(T)^*$,
  so $\theta(T)$ is a column contraction.
\end{proof}

\subsection{Factorization}

Our next goal is the factorization step in the Schur algorithm.
If $\varphi \in H^\infty$ with $\varphi(0) = 0$, then we may define $\psi = \varphi / z$,
so that $\psi \in H^\infty$ with $\varphi = z \psi$ and $\|\psi\|_\infty = \|\varphi\|_\infty$.
A generalization of this fact to the Drury--Arveson space was proved by Greene, Richter and Sundberg in
\cite[Corollary 4.2]{GRS05}. Using a version of Leech's theorem, they showed that if $\varphi \in \Mult(H^2_d)$ with $\varphi(0) = 0$,
then there exist $\psi_1,\ldots,\psi_d \in \Mult(H^2_d)$ so that $ \varphi = \sum_{j=1}^d z_j \psi_j$.
Moreover, one can achieve that the column norm of $(\psi_1,\ldots,\psi_d)$ is at most the multiplier norm
of $\varphi$.
The next proposition is a generalization of this factorization result to columns of multipliers
and to restrictions of $H^2_d$.
The proof of \cite[Corollary 4.2]{GRS05} carries over with minimal changes.

Let $\mathbf{z} =
\begin{bmatrix}
  z_1 & \cdots & z_d
\end{bmatrix}$ denote the row vector of coordinate functions.

\begin{prop}
  \label{prop:factorization}
  Let $F \subset \mathbb{B}_d$ with $0 \in F$ and let $\varphi_1,\ldots,\varphi_N \in \Mult(H^2_d \big|_F)$
  such that $\varphi_i(0) = 0$ for all $i$. Then there exist
  $\psi_{ij} \in \Mult(H^2_d \big|_F)$ for $1 \le i \le d, 1 \le j \le N$,
  so that
  \begin{equation*}
    \begin{bmatrix}
      \varphi_1 \\ \vdots \\ \varphi_N
    \end{bmatrix}
    =
    \begin{bmatrix}
      \mathbf{z} & 0 & \cdots & 0 \\
      0 & \mathbf{z} & \cdots & 0 \\
      \vdots & \ddots & \ddots & \vdots \\
      0 & 0 & \cdots & \mathbf{z}
    \end{bmatrix}
    \begin{bmatrix}
      \psi_{11} \\ \vdots \\ \psi_{d 1} \\ \psi_{12} \\ \vdots \\ \psi_{d N}
    \end{bmatrix}
  \end{equation*}
  and
  \begin{equation*}
    \left\| \begin{bmatrix}
      \psi_{11} \\ \vdots \\ \psi_{d 1} \\ \psi_{12} \\ \vdots \\ \psi_{d N}
    \end{bmatrix} \right\|_{\Mult(H^2_d \big|_F, H^2_d \big|_F \otimes \mathbb{C}^{d N})}
    = \left\|
    \begin{bmatrix}
      \varphi_1 \\ \vdots \\ \varphi_N
    \end{bmatrix} \right\|_{\Mult(H^2_d \big|_F, H^2_d \big|_F \otimes \mathbb{C}^{N})}.
  \end{equation*}
  
\end{prop}

\begin{proof}
  To shorten notation, set $\mathcal{H} = H^2_d \big|_F$ and $\mathcal{H}_0 = \{ f \in \mathcal{H}: f(0) = 0 \}$.
  Let
  \begin{equation*}
    \Phi =
    \begin{bmatrix}
      \varphi_1 \\ \vdots \\ \varphi_N
    \end{bmatrix} \in \Mult(\mathcal{H}, \mathcal{H} \otimes \mathbb{C}^N)
  \end{equation*}
  and assume without loss of generality that $\Phi$ has multiplier norm $1$.
  Since $\Phi(0) = 0$, $\Phi$ is in fact a contractive multiplier
  from $\cH$ to $\cH_0 \otimes \bC^N$. Note that the reproducing kernel
  of $\cH_0$ is
  \begin{equation*}
    \frac{1}{1 - \langle z , w \rangle} - 1 = \frac{\langle z , w \rangle}{1 - \langle z, w \rangle}.
  \end{equation*}
  Therefore, Lemma \ref{lem:multipliers_basic} implies that
  \begin{equation}
    \label{eqn:kernel_zero}
    \frac{1}{1 - \langle z, w \rangle} ( \langle z, w \rangle I_N - \Phi(z) \Phi(w)^*) \ge 0
  \end{equation}
  as a function of $(z,w)$ on $F \times F$.
  Consider the $B(\mathbb{C}^{d N}, \mathbb{C}^N)$-valued function
  \begin{equation*}
    \Theta(z) =
    \begin{bmatrix}
      \mathbf{z} & 0 & \cdots & 0 \\
      0 & \mathbf{z} & \cdots & 0 \\
      \vdots & \ddots & \ddots & \vdots \\
      0 & 0 & \cdots & \mathbf{z}
    \end{bmatrix}.
  \end{equation*}
  Then \eqref{eqn:kernel_zero} can be equivalently written as
  \begin{equation*}
    \frac{1}{1 - \langle z, w \rangle} ( \Theta(z) \Theta(w)^* - \Phi(z) \Phi(w)^* ) \ge 0.
  \end{equation*}
  In this setting, Leech's theorem for complete Pick spaces (see \cite[Theorem 8.57]{AM02})
  yields a contractive multiplier $\Psi \in \Mult(\mathcal{H}, \mathcal{H} \otimes \mathbb{C}^{dN})$ so that $\Phi = \Theta \Psi$. Since $\Theta$ is a contractive multiplier,
  we see that $\Psi$ has in fact multiplier norm equal to $1$.
  Writing $\Psi$ as the column of its coordinate functions finishes the proof.
\end{proof}

\subsection{From columns to matrices}

Suppose that we are given a factorization of a column multiplier as in Proposition \ref{prop:factorization}.
Then the corresponding row multiplier is given by
\begin{equation*}
  \begin{bmatrix}
    \varphi_1 & \cdots & \varphi_N
  \end{bmatrix} =
  \begin{bmatrix}
    z_1 & \cdots & z_d
  \end{bmatrix}
  \begin{bmatrix}
    \psi_{1 1} & \cdots & \psi_{1 N} \\
    \psi_{2 1} & \cdots & \psi_{2 N} \\
    \vdots  & \ddots & \vdots \\
    \psi_{d 1} & \cdots & \psi_{d N}
  \end{bmatrix}.
\end{equation*}

Therefore, we also have to study the process of going from columns of multipliers
to rectangular matrices. The Schur algorithm does not seem to be well suited
for this problem, as the automorphism groups of the unit balls
of $M_{d N,1}(\bC)$ and $M_{d,N}(\bC)$ are different.

Instead, we will make use of the following lemma, which crucially uses a result of Jury and Martin \cite{JM18}; see also \cite[Lemma 3.3]{CHar}.

\begin{lem}
  \label{lem:JM}
  Let $\cH$ be an irreducible complete Pick space and let $\Psi \in \Mult(\cH \otimes \bC^N, \cH \otimes \bC^M)$.
  Then
  \begin{equation*}
    \|\Psi\|_{\Mult(\cH \otimes \bC^N, \cH \otimes \bC^M)}
    = \sup \{ \|\Psi \Phi\|_{\Mult(\cH, \cH \otimes \bC^M)}:
    \Phi \in \Mult_1(\cH, \cH \otimes \bC^N) \}.
  \end{equation*}
\end{lem}

\begin{proof}
  By normalizing the reproducing kernel at a point, we may assume that $\mathcal{H}$ is normalized,
  see Subsection \ref{ss:CNP}.
  The inequality ``$\ge$'' in the lemma is trivial.

  To prove the reverse inequality, let $F \in \mathcal{H} \otimes \mathbb{C}^N$ with $\|F\| = 1$.
  Applying \cite[Theorem 1.1]{JM18}, we obtain $\Phi \in \Mult_1(\mathcal{H}, \mathcal{H} \otimes \mathbb{C}^N)$
  and $f \in \mathcal{H}$ with $\|f\| = 1$ and $F = \Phi f$. Thus,
  \begin{equation*}
    \|\Psi F\|_{\mathcal{H} \otimes \mathbb{C}^M} = \|\Psi \Phi f\|_{\mathcal{H} \otimes \mathbb{C}^M} \le \|\Psi \Phi\|_{\Mult(\mathcal{H}, \mathcal{H} \otimes \mathbb{C}^N)}.
  \end{equation*}
  Taking the supremum over all $F$ in the unit sphere of $\mathcal{H} \otimes \mathbb{C}^N$ yields the remaining inequality.
\end{proof}

With the help of Lemma \ref{lem:JM}, we can establish the key fact that the column-row property
implies the more general ``column-matrix property''.
For the Hardy space $H^2$, the multiplier norm of any matrix of multipliers
is simply the supremum of the pointwise operator norms. In this case, the following
result follows from the basic inequality $\|A\|_{op} \le \|A\|_{HS}$ between
the operator norm and the Hilbert--Schmidt norm of a matrix $A$.

\begin{prop}
  \label{prop:column_to_matrix}
  Let $\cH$ be an irreducible complete Pick space that satisfies the column-row property with constant $1$.
  Let $\psi_{i j} \in \Mult(\cH)$ for $1 \le i \le M$ and $1 \le j \le N$. Then
  \begin{equation*}
    \left\|
  \begin{bmatrix}
    \psi_{1 1} & \cdots & \psi_{1 N} \\
    \psi_{2 1} & \cdots & \psi_{2 N} \\
    \vdots  & \ddots & \vdots \\
    \psi_{M 1} & \cdots & \psi_{M N}
  \end{bmatrix}
     \right\|_{\Mult(\cH \otimes \bC^N, \cH \otimes \bC^M)}
    \le
    \left\|
    \begin{bmatrix}
      \psi_{11} \\ \vdots \\ \psi_{M 1} \\ \psi_{12} \\ \vdots \\ \psi_{M N}
    \end{bmatrix}
    \right\|_{\Mult(\cH, \cH \otimes \bC^{N M})}.
  \end{equation*}
\end{prop}

\begin{proof}
  We will compute the norm of the matrix on the left with the help of Lemma \ref{lem:JM}.
  To this end, let
  \begin{equation*}
    \Phi =
    \begin{bmatrix}
      \varphi_1 \\ \vdots \\ \varphi_N
    \end{bmatrix} \in \Mult_1(\mathcal{H}, \mathcal{H} \otimes\mathbb{C}^N).
  \end{equation*}
  Using commutativity of the multiplication in $\Mult(\mathcal{H})$, we find that
  \begin{gather*}
  \begin{bmatrix}
    \psi_{1 1} & \cdots & \psi_{1 N} \\
    \psi_{2 1} & \cdots & \psi_{2 N} \\
    \vdots  & \ddots & \vdots \\
    \psi_{M 1} & \cdots & \psi_{M N}
  \end{bmatrix}
    \begin{bmatrix}
      \varphi_1 \\ \varphi_2 \\ \vdots \\ \varphi_N
    \end{bmatrix} \\
    =
    \begin{bmatrix}
      \varphi_1 & \cdots & \varphi_N & 0 & \cdots & 0 & \cdots & 0 & \cdots & 0 \\
      0 & \cdots & 0 & \varphi_1 & \cdots & \varphi_N & \cdots & 0 & \cdots & 0 \\
      \vdots & \vdots & \vdots & \vdots & \vdots & \vdots & \ddots & \vdots & \vdots & \vdots \\
      0 & \cdots & 0 & 0 & \cdots & 0 & \cdots & \varphi_1 & \cdots & \varphi_N
    \end{bmatrix}
    \begin{bmatrix}
      \psi_{11} \\ \vdots \\ \psi_{1 N} \\ \psi_{21} \\ \vdots \\ \psi_{M N}
    \end{bmatrix}.
  \end{gather*}
  Here, the matrix on the left is the $M \times M N$ block diagonal matrix whose diagonal blocks
  are all equal to the row
  $
  \begin{bmatrix}
    \varphi_1 & \cdots & \varphi_N
  \end{bmatrix}$.
  Since $\mathcal{H}$ satisfies the column-row property with constant $1$ by assumption, the matrix on the left,
  being a direct sum of the rows
  $\begin{bmatrix}
    \varphi_1 & \cdots & \varphi_N
  \end{bmatrix}$, has multiplier norm at most $1$.
  Thus, Lemma \ref{lem:JM} implies that
  \begin{equation*}
    \left\|
  \begin{bmatrix}
    \psi_{1 1} & \cdots & \psi_{1 N} \\
    \psi_{2 1} & \cdots & \psi_{2 N} \\
    \vdots  & \ddots & \vdots \\
    \psi_{M 1} & \cdots & \psi_{M N}
  \end{bmatrix}
  \right\|_{\Mult}
    \le \left\|
    \begin{bmatrix}
      \psi_{11} \\ \vdots \\ \psi_{1 N} \\ \psi_{21} \\ \vdots \\ \psi_{M N}
    \end{bmatrix}
    \right\|_{\Mult}
    =
    \left\|
    \begin{bmatrix}
      \psi_{11} \\ \vdots \\ \psi_{M 1} \\ \psi_{12} \\ \vdots \\ \psi_{M N}
    \end{bmatrix}
    \right\|_{\Mult},
  \end{equation*}
  where in the last step, we used the elementary fact that the column norm is invariant under permutations.
\end{proof}

The following simple example shows that in general, it is not true that
a sequence of operators that forms both a row and a column contraction also forms
contractive matrices.

\begin{exa}
  Consider the scaled $2 \times 2$ matrix units
  \begin{equation*}
    E_{11} = \frac{1}{\sqrt{2}}
    \begin{bmatrix}
      1 & 0\\
      0 & 0
    \end{bmatrix},
    E_{12} =\frac{1}{\sqrt{2}}
    \begin{bmatrix}
      0 & 1\\
      0 & 0
    \end{bmatrix},
    E_{21} =\frac{1}{\sqrt{2}}
    \begin{bmatrix}
      0 & 0\\
      1 & 0
    \end{bmatrix},
    E_{22} =\frac{1}{\sqrt{2}}
    \begin{bmatrix}
      0 & 0\\
      0 & 1
    \end{bmatrix}.
  \end{equation*}
  One easily checks that
  \begin{equation*}
    \left\|
    \begin{bmatrix}
      E_{11} & E_{12} & E_{2 1} & E_{2 2}
    \end{bmatrix} \right\|
    = 1
    = \left\|
    \begin{bmatrix}
      E_{11} \\ E_{12} \\ E_{2 1} \\ E_{2 2}
    \end{bmatrix} \right\|,
  \end{equation*}
  yet
  \begin{equation*}
    \left\|
    \begin{bmatrix}
      E_{1 1} & E_{12} \\ E_{21} & E_{22}
    \end{bmatrix} \right\| = \sqrt{2}.
  \end{equation*}
\end{exa}

\subsection{Proof of the main result}

We are almost ready to put everything together to prove the main result.
We isolate one last lemma, which will be useful in the inductive proof.
It is an application of the Schur complement technique.

\begin{lem}
  \label{lem:Schur_complement}
  Let $F = E \cup \{0\} \subset \mathbb{B}_d$ be a finite set and let $\mathcal{E},\mathcal{F}$ be Hilbert spaces.
  \begin{enumerate}[label=\normalfont{(\alph*)}]
    \item 
  A function $\Phi: F \to B(\mathcal{E},\mathcal{F})$ with $\Phi(0) = 0$
  is a contractive multiplier of $H^2_d \big|_F$ if and only if
  \begin{equation*}
    (z,w) \mapsto \frac{1}{1 - \langle z,w \rangle} \Big( I_{\mathcal{F}} - \Phi(z) \Phi(w)^* \Big) - I_{\mathcal{F}}
    \ge 0
  \end{equation*}
  on $E \times E$.
\item Let $\Psi: F \to B(\mathcal{E},\mathbb{C}^d)$ be a function and suppose that $\Psi \big|_{E}$ is a contractive multiplier of $H^2_d \big|_{E}$. Then
  $\mathbf{z} \Psi$ is a contractive row multiplier of $H^2_d \big|_F$.
  \end{enumerate}
\end{lem}

\begin{proof}
  (a)
  Let $E = \{\lambda_1,\ldots,\lambda_{n-1} \}$ and
  set $\lambda_n = 0$, so that $F = \{\lambda_1,\ldots,\lambda_n\}$.
  An application of Lemma \ref{lem:multipliers_basic} shows that
  $\Phi$ is a contractive multiplier of $H^2_d \big|_F$ if and only if
  \begin{equation}
    \label{eqn:Schur_complement}
    \Big[
    \frac{1}{1 - \langle \lambda_i,\lambda_j \rangle} ( I_{\mathcal{F}}- \Phi(\lambda_i) \Phi(\lambda_j)^* ) \Big]_{i,j=1}^n \ge 0.
  \end{equation}
  Since $\Phi(0) = 0$, each entry in the last row and in the last column of this matrix is equal to $I_{\mathcal{F}}$.
  Taking the Schur complement of the lower right corner
  (see, for instance, \cite[Lemma 7.27]{AM02}), we see that \eqref{eqn:Schur_complement}
  holds if and only if
  \begin{equation*}
    \Big[ \frac{1}{1 - \langle \lambda_i,\lambda_j \rangle} ( I_{\mathcal{F}}- \Phi(\lambda_i) \Phi(\lambda_j)^* ) - I_{\mathcal{F}} \Big]_{i,j=1}^{n-1} \ge 0,
  \end{equation*}
  which proves part (a).

  (b)
  Since $\Psi|_E$ is a contractive multiplier of $H^2_d \big|_E$, we find that
  \begin{equation*}
    \frac{1}{1 - \langle z,w \rangle }( I_d - {\Psi}(z) {\Psi}(w)^*) \ge 0
  \end{equation*}
  on $E \times E$ by Lemma \ref{lem:multipliers_basic}.
  Multiplying this relation with $
  \begin{bmatrix}
    z_1 & \cdots & z_d
  \end{bmatrix}$
  on the left and with
  $\begin{bmatrix}
    w_1 & \cdots & w_d
  \end{bmatrix}^*$ on the right and using
  the identity $\frac{\langle z,w \rangle }{1 - \langle z,w \rangle} = \frac{1}{1 - \langle z,w \rangle} - 1$,
  we find with the help of Lemma \ref{lem:positive_conjugation} that
  \begin{equation*}
    \frac{1}{1 - \langle z,w \rangle}  \Big( 1 - \Phi(z) (\Phi(w))^* \Big) - 1 \ge 0
  \end{equation*}
  on $E \times E$.
  Thus, part (a) shows that $\Phi$ is a contractive row multiplier of $H^2_d \big|_F$.
\end{proof}

\begin{rem}
  The complete Pick property of $H^2_d$ can be used to obtain a shorter, but somewhat less explicit proof of part (b) of Lemma \ref{lem:Schur_complement}.
  Indeed, if $\Psi: F \to B(\mathcal{E},\mathbb{C}^d)$ is a function with the property
  that $\Psi \big|_{E}$ is a contractive multiplier of $H^2_d \big|_{E}$, then the complete Pick property of $H^2_d$
  implies that there exists a contractive multiplier $\widehat \Psi: \mathbb{B}_d \to B(\mathcal{\mathcal{E}},\mathbb{C}^d)$
  of $H^2_d$  satisfying $\widehat{\Psi} \big|_{E} = \Psi \big|_{E}$.
  Moreover, $\mathbf{z} \widehat{\Psi} \big|_{F} = \mathbf{z} \Psi$, as the two functions agree on $E$ and at the origin.
  Since $\mathbf{z}$ is a contractive row multiplier of $H^2_d$,
  the product $\mathbf{z} \widehat \Psi$ is contractive row multiplier of $H^2_d$,
  hence $\mathbf{z} \Psi = \mathbf{z} \widehat{\Psi} \big|_F$
  is a contractive row multiplier of $H^2_d \big|_F$.
\end{rem}

We are now ready to prove Theorem \ref{thm:main}, which we restate.

\begin{thm}
  Every normalized complete Pick space satisfies the column-row property with constant $1$.
\end{thm}

\begin{proof}
  By Lemma \ref{lem:reduction_finite}, it suffices to show that $H^2_d \big|_F$ satisfies
  the column-row property with constant $1$ for all $d \in \mathbb{N}$ and all finite sets $F \subset \mathbb{B}_d$.
  As remarked at the beginning of Subsection \ref{ss:red}, it is also sufficient
  to consider columns of a finite length $N \in \mathbb{N}$.
  We fix $d \in \mathbb{N}$ for the remainder of the proof
  and we prove the statement by induction on $n = |F|$.

  The base case $n=1$ is easy, as the multiplier norm
  in a reproducing kernel Hilbert space on a singleton is simply the norm at the singleton.
  More explicitly, suppose that $F = \{\lambda_1\}$ and let $\Phi$ be a contractive
  column multiplier of $H^2_d \big|_F$,
  say  $\Phi \in \Mult_1(H^2_d \big|_F, H^2_d \big|_F \otimes \mathbb{C}^N)$.
  Then
  \begin{equation*}
    \frac{1}{1 - \|\lambda_1\|^2} (I_N - \Phi(\lambda_1) \Phi(\lambda_1)^*) \ge 0
  \end{equation*}
  by Lemma \ref{lem:multipliers_basic},
  which is equivalent to saying that $\|\Phi(\lambda_1)\|_{B(\mathbb{C}, \mathbb{C}^N)} \le 1$.
  Thus, $\|\Phi^T(\lambda_1)\|_{B(\mathbb{C}^N,\mathbb{C})} \le 1$, so that
  \begin{equation*}
    \frac{1}{1 - \|\lambda_1\|^2} (1 - \Phi^T(\lambda_1) (\Phi^T(\lambda_1))^*) \ge 0
  \end{equation*}
  and hence $\Phi^T \in \Mult_1(H^2_d \big|_F \otimes \mathbb{C}^N, H^2_d \big|_F)$,
  again by Lemma \ref{lem:multipliers_basic}, i.e.\ $\Phi^T$ is a contractive row multiplier.
  This finishes the proof in the case $n=1$.

  Next, let $n \ge 2$ and suppose that we already know that for each subset $E \subset \bB_d$
  with $|E| \le n-1$, the space $H^2_d \big|_{E}$ has the column-row property
  with constant $1$.
  Let $F \subset \mathbb{B}_d$ with $|F| = n$
  and let $\Phi$ be a contractive column multiplier of $H^2_d \big|_F$, say
  $\Phi \in \Mult_1(H^2_d \big|_F, H^2_d \big|_F \otimes \mathbb{C}^N)$.
  Our goal is to show that $\Phi^T$ is a contractive row multiplier, i.e.\ that
  $\Phi^T \in \Mult_1(H^2_d \big|_F \otimes \mathbb{C}^N, H^2_d)$.

  First, we apply the conformal automorphism step of the Schur algorithm.
  Recall that $\Aut(\mathbb{B}_d)$ acts transitively on $\mathbb{B}_d$ (see \cite[Theorem 2.2.3]{Rudin08}),
  so there exists $\theta \in \Aut(\mathbb{B}_d)$ with $0 \in \theta^{-1}(F)$.
  Lemma \ref{lem:conformal_domain} therefore shows that by replacing $F$ with $\theta^{-1}(F)$
  and $\Phi$ with $\Phi \circ \theta$, we may assume without loss of generality that $0 \in F$.
  The fact that $\Phi$ is a contractive column
  multiplier implies that $\|\Phi(\lambda)\| \le 1$ for all $\lambda \in F$, just as in the proof of the base case $n=1$.
  Thus, Lemma \ref{lem:conformal_range} shows that by replacing $\Phi$ with $\theta \circ \Phi$
  for a suitable element of $\Aut(\mathbb{B}_N)$, we may assume that $\Phi(0) = 0$.

  Next, we apply the factorization step of the Schur algorithm.
  Let $\varphi_1,\ldots,\varphi_N$ be the coordinate functions of $\Phi$.
  In our setting, Proposition \ref{prop:factorization} implies that there
  exists a contractive column multiplier
  \begin{equation*}
    \Psi =
    \begin{bmatrix}
      \psi_{11} \\ \vdots \\ \psi_{d 1} \\ \psi_{12} \\ \vdots \\ \psi_{d N}
    \end{bmatrix} \in \Mult_1(H^2_d \big|_F, H^2_d \big|_F \otimes \mathbb{C}^{d N})
  \end{equation*}
  such that
  \begin{equation*}
    \Phi =
    \begin{bmatrix}
      \varphi_1 \\ \vdots \\ \varphi_N
    \end{bmatrix}
    =
    \begin{bmatrix}
      \mathbf{z} & 0 & \cdots & 0 \\
      0 & \mathbf{z} & \cdots & 0 \\
      \vdots & \ddots & \ddots & \vdots \\
      0 & 0 & \cdots & \mathbf{z}
    \end{bmatrix}
    \begin{bmatrix}
      \psi_{11} \\ \vdots \\ \psi_{d 1} \\ \psi_{12} \\ \vdots \\ \psi_{d N}
    \end{bmatrix}.
  \end{equation*}
  Let
  \begin{equation*}
    \widetilde \Psi
  = \begin{bmatrix}
    \psi_{1 1} & \cdots & \psi_{1 N} \\
    \psi_{2 1} & \cdots & \psi_{2 N} \\
    \vdots  & \ddots & \vdots \\
    \psi_{d 1} & \cdots & \psi_{d N}
  \end{bmatrix},
  \end{equation*}
  so that
  \begin{equation*}
    \Phi^T =
    \begin{bmatrix}
      \varphi_1 & \cdots & \varphi_N
    \end{bmatrix}
    = \mathbf{z} \widetilde \Psi.
  \end{equation*}

  To apply the inductive hypothesis, let us write $F = E \cup \{0\}$, where $|E| = n-1$.
  Notice that $\Psi \big|_{E}$ is in particular a contractive column multiplier of $H^2_d \big|_{E}$.
  By induction hypothesis, the irreducible complete Pick space
  $H^2_d \big|_{E}$ satisfies the column-row property with constant $1$.
  Proposition \ref{prop:column_to_matrix} therefore shows that the matrix $\widetilde \Psi \big|_{E}$ is
  a contractive $B(\mathbb{C}^N, \mathbb{C}^d)$-valued multiplier of $H^2_d \big|_{E}$.
  Lemma \ref{lem:Schur_complement} (b) now implies that $\Phi^T = \mathbf{z} \widetilde \Psi$
  is a contractive row multiplier of $H^2_d \big|_F$, as desired.
\end{proof}

Let us illustrate the proof above with a simple example in which the row norm is strictly
smaller than the column norm.

\begin{exa} Let $d=2$ and consider
  \begin{equation*}
    \Phi
    = \frac{1}{\sqrt{2}}
    \begin{bmatrix}
      z_1 \\ z_2
    \end{bmatrix}.
  \end{equation*}
  Then $\|\Phi\|_{\Mult(H^2_2, H^2_d \otimes \mathbb{C}^2)} = 1$, whereas
  $\|\Phi^T\|_{\Mult(H^2_2 \otimes \mathbb{C}^2, H^2_2)} = \frac{1}{\sqrt{2}}$.
  Indeed, it is a fundamental property of the Drury--Arveson space that
  the coordinate functions form a row contraction, so by the basic $\sqrt{n}$-bound
  mentioned in the introduction,
  \begin{equation*}
    \|\Phi\|_{\Mult} \le \sqrt{2} \|\Phi^T\|_{\Mult} \le 1,
  \end{equation*}
  and applying $\Phi$ to the constant function $1 \in H^2_2$ shows that
  $\|\Phi\|_{\Mult} \ge 1$, so equality holds throughout.

  Notice that $\Phi(0) = 0$, so we may apply Proposition \ref{prop:factorization}.
  In this case, a factorization is simply
  \begin{equation*}
    \Phi =
    \begin{bmatrix}
      z_1 & z_2 & 0 & 0 \\
      0 & 0 & z_1 & z_2
    \end{bmatrix}
    \begin{bmatrix}
      \frac{1}{\sqrt{2}} \\ 0 \\ 0 \\ \frac{1}{\sqrt{2}}
    \end{bmatrix}.
  \end{equation*}
  Proceeding as in the above proof, we see that
  \begin{equation*}
    \Phi^T =
    \begin{bmatrix}
      z_1 & z_2
    \end{bmatrix}
    \begin{bmatrix}
      \frac{1}{\sqrt{2}} & 0 \\
      0 & \frac{1}{\sqrt{2}}
    \end{bmatrix}.
  \end{equation*}
  Observe that the square matrix in the
  factorization of $\Phi^T$ has norm $\frac{1}{\sqrt{2}}$, whereas the
  column in the factorization of $\Phi$ has norm $1$.
  Thus, we can see the decrease in multiplier norm in the factorization in this case.
  In this simple example, there is no need to restrict
  to finite subsets of the ball, as the multiplier becomes constant after one step of the Schur
  algorithm. Reversing the steps above, we also see where the argument breaks down
  when going from rows to columns.
%
\end{exa}

Combining Theorem \ref{thm:main} with Proposition \ref{prop:column_to_matrix},
it follows immediately that complete Pick spaces in fact satisfy the ``column-matrix property''.

\begin{cor}
  \label{cor:column-matrix}
  Let $\mathcal{H}$ be a normalized complete Pick space and let
  let $\psi_{i j} \in \Mult(\cH)$ for $1 \le i \le M$ and $1 \le j \le N$. Then
  \begin{equation*}
    \left\|
  \begin{bmatrix}
    \psi_{1 1} & \cdots & \psi_{1 N} \\
    \psi_{2 1} & \cdots & \psi_{2 N} \\
    \vdots  & \ddots & \vdots \\
    \psi_{M 1} & \cdots & \psi_{M N}
  \end{bmatrix}
     \right\|_{\Mult(\cH \otimes \bC^N, \cH \otimes \bC^M)}
    \le
    \left\|
    \begin{bmatrix}
      \psi_{11} \\ \vdots \\ \psi_{M 1} \\ \psi_{12} \\ \vdots \\ \psi_{M N}
    \end{bmatrix}
    \right\|_{\Mult(\cH, \cH \otimes \bC^{N M})}. \qed
  \end{equation*}
\end{cor}

\subsection{Spaces with a complete Pick factor}
\label{ss:CNP_factor}

In recent years, several results about complete Pick spaces have been generalized to spaces
whose reproducing kernel has a complete Pick factor; see for instance \cite{AHM+17c,CHS20} and \cite[Section 4]{AHM+17}.
We show that the column-row property for complete Pick spaces generalizes in a similar fashion.

Throughout this subsection, we will assume the following setting. Let $\mathcal{H}_K$ and $\mathcal{H}_S$ be two reproducing kernel Hilbert spaces on $X$
with reproducing kernels $K$ and $S$, respectively. Assume that $\mathcal H_S$ is a normalized complete Pick space and that $K/S \ge 0$.
Then the positive kernel $K/S$ may be factored as
\begin{equation*}
  (K/S)(z,w) = G(z) G(w)^*,
\end{equation*}
where $G: X \to B(\mathcal{G},\mathbb{C})$ for some auxiliary Hilbert space $\mathcal{G}$.

A basic example of this setting occurs when $\mathcal{H}_S$ is the Hardy space and $\mathcal{H}_K$ is the Bergman space on the disc.
In the references cited above,
results about $\Mult(\mathcal{H}_S)$ were generalized to $\Mult(\mathcal{H}_S,\mathcal{H}_K)$.
The following lemma makes it possible to carry out this generalization for the column-row property;
it is essentially a vector valued version of \cite[Proposition 4.10]{AHM+17}.

\begin{lem}
  \label{lem:pair_factor}
  Assume the setting above and let $\mathcal{E},\mathcal{F}$ be Hilbert spaces.
  The following are equivalent for a function $\Phi: X \to \mathcal{B}(\mathcal{E},\mathcal{F})$:
  \begin{enumerate}[label=\normalfont{(\roman*)}]
    \item The function $\Phi$ is a contractive multiplier from $\mathcal{H}_S \otimes \mathcal{E}$ to $\mathcal{H}_K \otimes \mathcal{F}$.
    \item There exists a contractive multiplier $\Psi$ from $\mathcal{H}_S \otimes \mathcal{E}$ to $\mathcal{H}_S \otimes \mathcal{G} \otimes \mathcal{F}$
      so that $\Phi = (G \otimes I_{\mathcal{F}}) \Psi$.
  \end{enumerate}
\end{lem}

\begin{proof}
  (ii) $\Rightarrow$ (i) This is immediate from the fact that $G$ is a contractive multiplier
  from $\mathcal{H}_S \otimes \mathcal{G}$ to $\mathcal{H}_K$, which in turn follows from Lemma \ref{lem:multipliers_basic}.

  (i) $\Rightarrow$ (ii) Let $\widetilde G(z) = G(z) \otimes I_{\mathcal{F}}$. The definition of $G$ and Lemma \ref{lem:multipliers_basic} show that
  \begin{equation*}
    S(z,w) (\widetilde G(z) \widetilde G(w) -  \Phi(z) \Phi(w)^*)
    = K(z,w) I_{\mathcal{F}} - S(z,w) \Phi(z) \Phi(w)^* \ge 0
  \end{equation*}
  as a function of $(z,w)$. Since $\mathcal{H}_S$ is a complete Pick space, Leech's theorem
  \cite[Theorem 8.57]{AM02} yields the desired multiplier $\Psi$ of $\mathcal{H}_S$.
\end{proof}

We are now ready to generalize the column-row property and also the column-matrix property to pairs of spaces.
The column-row property corresponds to the case $M=1$ below.

\begin{thm}
  \label{thm:column_matrix_pairs}
  Assume the setting of Subsection \ref{ss:CNP_factor}.
  Let $\varphi_{i j} \in \Mult(\cH_S,\mathcal{H}_K)$ for $1 \le i \le M$ and $1 \le j \le N$. Then
  \begin{equation*}
    \left\|
  \begin{bmatrix}
    \varphi_{1 1} & \cdots & \varphi_{1 N} \\
    \varphi_{2 1} & \cdots & \varphi_{2 N} \\
    \vdots  & \ddots & \vdots \\
    \varphi_{M 1} & \cdots & \varphi_{M N}
  \end{bmatrix}
     \right\|_{\Mult(\cH_S \otimes \bC^N, \cH_K \otimes \bC^M)}
    \le
    \left\|
    \begin{bmatrix}
      \varphi_{11} \\ \vdots \\ \varphi_{M 1} \\ \varphi_{12} \\ \vdots \\ \varphi_{M N}
    \end{bmatrix}
    \right\|_{\Mult(\cH_S, \cH_K \otimes \bC^{N M})}.
  \end{equation*}
\end{thm}

\begin{proof}
  Let $\Phi$ be the column of the $\varphi_{ij}$ on the right, and assume that $\Phi$ has multiplier norm $1$.
  The implication (i) $\Rightarrow$ (ii) of Lemma \ref{lem:pair_factor} yields a contractive multiplier $\Psi$
  from $\mathcal{H}_S$ to $\mathcal{H}_S \otimes \mathcal{G} \otimes \mathbb{C}^{MN}$ so that
  $\Phi = (G \otimes I_{MN}) \Psi$. Write $\Psi$ as a column of multipliers $\Psi_{ij}$ from $\mathcal{H}_S$
  to $\mathcal{H}_S \otimes \mathcal{G}$ so that $\varphi_{ij} = G \Psi_{i j}$ for all $1 \le i \le M, 1 \le j \le N$.
  Then
  \begin{equation*}
  \begin{bmatrix}
    \varphi_{1 1} & \cdots & \varphi_{1 N} \\
    \varphi_{2 1} & \cdots & \varphi_{2 N} \\
    \vdots  & \ddots & \vdots \\
    \varphi_{M 1} & \cdots & \varphi_{M N}
  \end{bmatrix}
  =
  \begin{bmatrix}
    G & 0 & \cdots & 0 \\
    0 & G & \cdots & 0 \\
    \vdots & \ddots & \ddots & \vdots \\
    0 & 0 & \cdots & G
  \end{bmatrix}
  \begin{bmatrix}
    \Psi_{1 1} & \cdots & \Psi_{1 N} \\
    \Psi_{2 1} & \cdots & \Psi_{2 N} \\
    \vdots  & \ddots & \vdots \\
    \Psi_{M 1} & \cdots & \Psi_{M N}
  \end{bmatrix}.
  \end{equation*}
  Since $\Psi$ is a contractive multiplier of $\mathcal{H}_S$, the column-matrix property
  of $\mathcal{H}_S$ (Corollary \ref{cor:column-matrix}), combined with an obvious approximation argument
  to approximate each $\Psi_{ij}$ by a finite column, shows that the matrix on the right is a contractive
  multiplier of $\mathcal{H}_S$. Thus, the trivial direction (ii) $\Rightarrow$ (i) of Lemma \ref{lem:pair_factor}
  shows that the matrix on the left is a contractive multiplier from $\mathcal{H}_S$ to $\mathcal{H}_K$.
\end{proof}

Note that in the proof above, we used the column-matrix property of $\mathcal{H}_S$ even in the case $M=1$,
i.e. when only showing the column-row property for the pair $(\mathcal{H}_S,\mathcal{H}_K)$.

\section{Further applications}
\label{sec:applications}

\subsection{Weak product spaces}

We already observed in the introduction that combining the main result, Theorem \ref{thm:main}, with known results
in the literature yields several results about weak product spaces.
We now collect a few more of these consequences.

The \emph{Smirnov class} of a normalized complete Pick space $\mathcal{H}$ is defined to be
\begin{equation*}
  N^+(\mathcal{H}) = \Big\{ \frac{\varphi}{\eta}: \varphi, \eta \in \Mult(\mathcal{H}), \eta \text{ cyclic} \Big\}.
\end{equation*}
Recall that $\eta \in \Mult(\mathcal{H})$ is said to be \emph{cyclic} if the multiplication operator $M_{\eta}$
on $\mathcal{H}$ has dense range. In \cite{AHM+17a}, it was shown that $\mathcal{H} \subset N^+(\mathcal{H})$
for any normalized complete Pick space $\mathcal{H}$.
Moreover, \cite[Corollary 3.4]{AHM+18} shows that if $\mathcal{H}$ satisfies
the column-row property, then $\mathcal{H} \odot \mathcal{H} \subset N^+(\mathcal{H})$.
Thus, in combination with Theorem \ref{thm:main}, we obtain this inclusion for all normalized complete Pick spaces.
Notice that since $N^+(\mathcal{H})$ is an algebra, the inclusion
$\mathcal{H} \odot \mathcal{H} \subset N^+(\mathcal{H})$ also follows from the description
$\mathcal{H} \odot \mathcal{H} = \{f \cdot g: f, g \in \mathcal{H}\}$ of Theorem \ref{thm:wp_factorization}.
In fact, \cite[Theorem 3.3]{AHM+18} yields more precise information.
To put the next result into perspective, it is useful to recall that if $\psi \in \Mult(\mathcal{H})$
with $\psi \neq 1$, then $1 - \psi$ and $(1- \psi)^2$ are cyclic, see \cite[Lemma 2.3]{AHM+17a}.

\begin{thm}
  Let $\mathcal{H}$ be a complete Pick space that is normalized at $z_0$
  and let $h \in \mathcal{H} \odot \mathcal{H}$
  with $\|h\|_{\mathcal{H} \odot \mathcal{H}} \le 1$. Then there exist $\varphi,\psi \in \Mult(\mathcal{H})$
  with $\|\varphi\|_{\Mult(\mathcal{H})} \le 1, \|\psi\|_{\Mult(\mathcal{H})} \le 1$, $\psi(z_0) = 0$
  such that
  \begin{equation*}
    h = \frac{\varphi}{(1 - \psi)^2}.
  \end{equation*}
\end{thm}

\begin{proof}
  This is the statement of \cite[Theorem 3.3]{AHM+18}, the only difference being that in \cite[Theorem 3.3]{AHM+18},
  one has $\varphi \in \Mult(\mathcal{H} \odot \mathcal{H})$ with $\|\varphi\|_{\Mult(\mathcal{H} \odot \mathcal{H})}\le 1$. Thus, the result follows, for instance, from the equality $\Mult(\mathcal{H} \odot \mathcal{H}) = \Mult(\mathcal{H})$ of Theorem \ref{thm:wp_mult}.

  Alternatively, examination of the proof of Lemma 3.2 and Theorem 3.3 of \cite{AHM+18} shows
  that if $\mathcal{H}$ satisfies the column-row property with constant $1$, then one
  obtains that $\varphi \in \Mult(\mathcal{H})$ with $\|\varphi\|_{\Mult(\mathcal{H})} \le 1$,
  so the result also follows directly with the help of Theorem \ref{thm:main} independently of the results of \cite{CHar}.
\end{proof}

Nehari's theorem shows that the dual space of $H^1$ is the space of all symbols of bounded Hankel operators
on $H^2$. This can be generalized to weak product spaces of normalized complete Pick spaces $\mathcal{H}$.
In \cite[Subsection 2.2]{AHM+18}, a space $\Han(\mathcal{H})$ of symbols of bounded Hankel
operators on $\mathcal{H}$ is defined, and it is shown that the dual space of $\mathcal{H} \odot \mathcal{H}$
is isomorphic via a conjugate linear isometry to $\Han(\mathcal{H})$.
The definition of $\Han(\mathcal{H})$ is somewhat involved.
More concrete is the space $\mathcal{X}(\mathcal{H})$ of all those $b \in \mathcal{H}$
for which the densely defined bilinear form
\begin{equation*}
  \mathcal{H} \times \mathcal{H} \to \mathbb{C}, \quad
  (\varphi, f) \mapsto \langle \varphi f , b \rangle \quad (\varphi \in \Mult(\mathcal{H}), f \in \mathcal{H}),
\end{equation*}
is bounded, equipped with the norm of the bilinear form.
(Recall that the kernel functions of a normalized complete Pick space are multipliers,
hence $\Mult(\mathcal{H})$ is densely contained in $\mathcal{H}$.)
Explicitly, $b \in \mathcal{X}(\mathcal{H})$ if and only if there exists $C \ge 0$
so that
\begin{equation*}
  |\langle \varphi f ,b \rangle | \le C \|\varphi\|_{\mathcal{H}} \|f\|_{\mathcal{H}}
  \quad \text{ for all} \quad
\varphi \in \Mult(\mathcal{H}), f \in \mathcal{H}.
\end{equation*}
In \cite[Theorem 2.6]{AHM+18}, it is shown that in the presence
of the column-row property, one has $\Han(\mathcal{H}) = \mathcal{X}(H)$.
Thus, in combination with Theorem \ref{thm:main}, we obtain the following
version of Nehari's theorem.

\begin{thm}
  Let $\mathcal{H}$ be a normalized complete Pick space. Then there is a conjugate
  linear isometric isomorphism $(\mathcal{H} \odot \mathcal{H})^* = \mathcal{X}(\mathcal{H})$.
  On the dense subspace $\mathcal{H}$ of $\mathcal{H} \odot \mathcal{H}$,
  the action of an element $b \in \mathcal{X}(\mathcal{H})$ on $f \in \mathcal{H}$
  is given by $\langle f,b \rangle_{\mathcal{H}}$. \qed
\end{thm}

If $b \in \Han(\mathcal{H}) = \mathcal{X}(\mathcal{H})$, then the associated Hankel operator $H_b$
is the unique bounded linear operator $H_b: \mathcal{H} \to \overline{\mathcal{H}}$ satisfying
\begin{equation*}
  \langle H_b f, \overline{\varphi} \rangle_{\overline{\mathcal{H}}}
  = \langle \varphi f,b \rangle_{\mathcal{H}} \quad (\varphi \in \Mult(\mathcal{H}), f \in \mathcal{H}),
\end{equation*}
see \cite[Subsection 2.2]{AHM+18}.
Here $\overline{\mathcal{H}}$ denotes the conjugate Hilbert space.
It is easy to see that the kernel of every Hankel operator is multiplier invariant, i.e.\
invariant under $M_\varphi$ for all $\varphi \in \Mult(\mathcal{H})$.
In \cite[Corollary 3.8]{AHM+18}, it was shown that conversely,
every closed multiplier invariant subspace of $\mathcal{H}$ is an intersection of kernels of Hankel operators,
provided that $\mathcal{H}$ satisfies the column-row property. Thus, we obtain the following consequence.

\begin{thm}
  Let $\mathcal{H}$ be a normalized complete Pick space and let $\mathcal{M} \subset \mathcal{H}$ be a closed
  multiplier invariant subspace. Then there exists a sequence $(b_n)$ in $\mathcal{X}(\mathcal{H})$ so that
  \begin{equation*}
    \pushQED{\qed}
    \mathcal{M} = \bigcap_n \ker H_{b_n}. \qedhere
    \popQED
  \end{equation*} 
\end{thm}

\subsection{Interpolating sequences}

Let $\mathcal{H}$ be a normalized complete Pick space on $X$ with reproducing kernel $K$.
A sequence $(z_n)$ in $X$ is said to
\begin{itemize}
  \item[(IS)] be
    an \emph{interpolating sequence} for $\Mult(\mathcal{H})$ if
the evaluation map
\begin{equation*}
  \Mult(\mathcal{H}) \to \ell^\infty, \quad \varphi \mapsto (\varphi(z_n)),
\end{equation*}
is surjective;
\item[(C)]
  satisfy the \emph{Carleson measure condition} if there exists $C \ge 0$ so that
\begin{equation*}
  \sum_{n} \frac{|f(z_n)|^2}{K(z_n,z_n)} \le C \|f\|^2_{\mathcal{H}}
  \quad \text{ for all } f \in \mathcal{H};
\end{equation*}
\item[(WS)] be \emph{weakly separated} if there exists $\varepsilon > 0$ so that that for all $n \neq m$, there exists $\varphi \in \Mult_1(\mathcal{H})$ with $\varphi(z_n) = \varepsilon$
and $\varphi(z_m) = 0$.
\end{itemize}
The weak separation condition can be rephrased in terms of a (pseudo-)metric derived from the reproducing kernel $K$.
For this translation and background on interpolating sequences, see \cite[Chapter 9]{AM02}.
Carleson showed that for $\mathcal{H} = H^2$
a sequence satisfies (IS) if and only if it satisfies (C) and (WS).
This result was extended to all normalized complete Pick spaces in \cite{AHM+17}, using
the solution of the Kadison--Singer problem due to Marcus, Spielman and Srivastava \cite{MSS15}.
In \cite[Remark 3.7]{AHM+17}, it was observed that if $\mathcal{H}$ satisfies the column-row
property, then a simpler proof is possible. For the convenience of the reader,
we present the relevant part of the argument in which the column-row property enters.

\begin{thm}
  In every normalized complete Pick space $\mathcal{H}$, the equivalence
  (IS) $\Leftrightarrow$ (C) + (WS) holds. In this case, there exists a bounded linear right-inverse
  of the evaluation map $\Mult(\mathcal{H}) \to \ell^\infty$.
\end{thm}

\begin{proof}
  It is well known that every interpolating sequence is weakly separated and satisfies
  the Carleson measure condition; see \cite[Chapter 9]{AM02}.

  Conversely, suppose that $(z_n)$ is weakly separated and satisfies the Carleson measure condition.
  A theorem of Agler and \mcc\ \cite[Theorem 9.46 (c)]{AM02} shows that
  there exists a bounded column multiplier $\Phi \in \Mult(\mathcal{H}, \mathcal{H} \otimes \ell^2)$
  so that $\Phi(z_n) = e_n$, the $n$-th standard basis vector of $\ell^2$, for all $n \in \mathbb{N}$.
  Theorem \ref{thm:main} implies that the transposed multiplier $\Phi^T \in \Mult(\mathcal{H} \otimes \ell^2, \mathcal{H})$ is bounded as well. Let $\Delta: \ell^\infty \to B(\ell^2)$ be the embedding via diagonal operators and
  define
  \begin{equation*}
    T: \ell^\infty \to \Mult(\mathcal{H}), \quad w \mapsto \Phi^T \Delta(w) \Phi.
  \end{equation*}
  Observe that if $w \in \ell^\infty$, then $T(w)$ is indeed a multiplication operator, and that
  \begin{equation*}
    T(w)(z_n) = e_n^T \Delta(w) e_n = w_n
  \end{equation*}
  for all $n \in \mathbb{N}$. Thus, $(z_n)$ satisfies (IS), and $T$ is the desired right-inverse of the evaluation map.
\end{proof}

\subsection{de Branges--Rovnyak spaces and extreme points}

In this subsection, we prove Theorem \ref{thm:extreme_point}.
We require the following special case of a result of Jury and Martin \cite{JuryPC,Jury17}
regarding extreme points of the unit ball of an operator algebra.
The proof below is a simplification of their proof.

\begin{lem}[Jury--Martin]
  \label{lem:JM_extreme}
  Let $\mathcal{H}$ be a Hilbert space, let $A,B \in B(\mathcal{H})$ and suppose that
  \begin{equation*}
    \Big\|
  \begin{bmatrix}
    B \\ A
  \end{bmatrix} \Big\| \le 1
  \quad \text{ and } \quad
  \|
  \begin{bmatrix}
    B & A
  \end{bmatrix} \| \le 1.
  \end{equation*}
  Then $\|B \pm \frac{1}{2} A^2 \| \le 1$.
\end{lem}

\begin{proof}
  Since the column and the row have norm at most one, so do
  \begin{equation*}
    X =
    \begin{bmatrix}
      1 & 0 & 0 \\
      0 & B & A
    \end{bmatrix}
    \quad
    \text{ and }
    \quad
    Y =
    \begin{bmatrix}
      0 & B \\
      1 & 0 \\
      0 &  \pm A
    \end{bmatrix}.
  \end{equation*}
  Hence
  \begin{equation*}
    \begin{bmatrix}
      \frac{1}{\sqrt{2}} & \frac{1}{\sqrt{2}}
    \end{bmatrix}
    X Y
    \begin{bmatrix}
      \frac{1}{\sqrt{2}} \\ \frac{1}{\sqrt{2}}
    \end{bmatrix}
    = \frac{1}{2}
    \begin{bmatrix}
      1 & 1
    \end{bmatrix}
    \begin{bmatrix}
      0 & B \\
      B & \pm A^2
    \end{bmatrix}
    \begin{bmatrix}
      1 \\ 1
    \end{bmatrix}
    = B \pm \frac{1}{2} A^2
  \end{equation*}
  has norm at most $1$ as well.
\end{proof}

We can now prove Theorem \ref{thm:extreme_point}, which we restate for the reader's convenience.
Recall that a multiplier $b$ of a reproducing kernel Hilbert space $\mathcal{H}$
of norm at most $1$ is said to be column extreme if there does not exist $a \in \Mult(\mathcal{H}) \setminus \{0\}$ so that
\begin{equation*}
  \Big\|
  \begin{bmatrix}
    b \\ a
  \end{bmatrix} \Big\|_{\Mult(\mathcal{H},\mathcal{H} \otimes \mathbb{C}^2)} \le 1.
\end{equation*}

\begin{thm}
  Let $\mathcal{H}$ be a normalized complete Pick space and let $b$ belong to the closed unit ball of $\Mult(\mathcal{H})$.
  Then $b$ is an extreme point of the closed unit ball of $\Mult(\mathcal{H})$ if and only if $b$ is column extreme.
\end{thm}

\begin{proof}
  The ``if'' part was already shown by Jury and Martin \cite{JM16}.
  We provide a variant of their argument in the spirit of the proof of Lemma \ref{lem:JM_extreme}.
  Suppose that $b$ is not an extreme point of the closed unit ball of $\Mult(\mathcal{H})$. Then there exists $a \in \Mult(\mathcal{H}) \setminus \{0\}$
  so that $\|b \pm a\|_{\Mult(\mathcal{H})} \le 1$. Therefore,
  \begin{equation*}
    \begin{bmatrix}
      b \\ a
    \end{bmatrix}
    =
    \begin{bmatrix}
    \frac{1}{\sqrt{2}} & 
    \frac{1}{\sqrt{2}} \\
    \frac{1}{\sqrt{2}} &
    - \frac{1}{\sqrt{2}}
    \end{bmatrix}
    \begin{bmatrix}
      b + a & 0 \\ 0 & b - a
    \end{bmatrix}
    \begin{bmatrix}
      \frac{1}{\sqrt{2}} \\ \frac{1}{\sqrt{2}}
    \end{bmatrix}
  \end{equation*}
  has multiplier norm at most $1$, because the scalar square matrix is unitary. Hence $b$ is not column extreme.

  Conversely, suppose that $b$ is not column extreme.
  By definition, there exists $a \in \Mult(\mathcal{H}) \setminus \{0\}$ so that
  \begin{equation*}
    \Big\|
    \begin{bmatrix}
      b \\ a
    \end{bmatrix} \Big\|_{\Mult(\mathcal{H},\mathcal{H} \otimes \mathbb{C}^2)} \le 1.
  \end{equation*}
  Theorem \ref{thm:main} implies that
  \begin{equation*}
    \|
    \begin{bmatrix}
      b & a
    \end{bmatrix} \|_{\Mult(\mathcal{H} \otimes \mathbb{C}^2, \mathcal{H})} \le 1,
  \end{equation*}
  hence $\|b \pm \frac{1}{2} a^2\|_{\Mult(\mathcal{H})} \le 1$ by Lemma \ref{lem:JM_extreme}.
  Since $\Mult(\mathcal{H})$ is an algebra of functions, it does not contain
  any non-zero nilpotent elements, so $a^2 \neq 0$. Thus, $b$ is not an extreme point of the closed unit ball
of $\Mult(\mathcal{H})$.
\end{proof}

\section{Counterexamples and questions}

\subsection{Failure of the complete column-row property}

We briefly discuss how Theorem \ref{thm:main} (the column-row property) and Corollary \ref{cor:column-matrix} (the column-matrix property)
can be interpreted in the theory of operator spaces and how a natural generalization of these results fails.

Let $N \in \mathbb{N}$ and let
\begin{equation*}
  \mathcal{M}^C = \Mult(\mathcal{H}, \mathcal{H} \otimes \mathbb{C}^N)
  \quad \text{ and } \quad
  \mathcal{M}^R = \Mult(\mathcal{H} \otimes \mathbb{C}^N, \mathcal{H})
\end{equation*}
be the space of column, respectively row, multipliers with $N$ components, both equipped
with the multiplier norm.
(One could also allow infinite columns and rows, but we restrict to columns and rows of a fixed finite length
for simplicity.)
Theorem \ref{thm:main} shows that the transpose mapping $T: \mathcal{M}^C \to \mathcal{M}^R$ is contractive.

Identifying a multiplier with its multiplication operator, we see that
$\mathcal{M}^C$ and $\mathcal{M}^R$ are in fact concrete operator spaces.
In particular, for each $n,m \in \mathbb{N}$, there is a natural norm
on $M_{n,m}(\mathcal{M}^C)$ and $M_{n,m}(\mathcal{M}^R)$, via
the identifications
\begin{equation*}
  M_{n,m}(\mathcal{M}^C)= \Mult(\mathcal{H} \otimes \mathbb{C}^m, \mathcal{H} \otimes \mathbb{C}^N \otimes \mathbb{C}^n)
\end{equation*}
and
\begin{equation*}
  M_{n,m}(\mathcal{M}^R)= \Mult(\mathcal{H} \otimes \mathbb{C}^N \otimes \mathbb{C}^m, \mathcal{H} \otimes \mathbb{C}^n).
\end{equation*}
It is not hard to see that Corollary \ref{cor:column-matrix} is equivalent to the assertion that for each $n \in \mathbb{N}$, the induced map
\begin{equation*}
  T^{(n,1)} : M_{n,1}(\mathcal{M}^C) \to M_{n,1}(\mathcal{M}^R),
\end{equation*}
defined by applying $T$ entrywise, is contractive.

But if $N \ge 2$, it is very easy to see that $T$ is not completely contractive, i.e.\ that
\begin{equation*}
  T^{(n,m)}: M_{n,m}(\mathcal{M}^C) \to M_{n,m}(\mathcal{M}^R)
\end{equation*}
is not contractive for all $n,m \in \mathbb{N}$.
This is simply because of the failure of complete contractivity of the ordinary transpose map.
Indeed, suppose for simplicity that $N=2$ and consider the element
\begin{equation*}
  \begin{bmatrix}
    \begin{bmatrix}
      1 \\ 0
    \end{bmatrix}
    &
    \begin{bmatrix}
      0 \\ 1
    \end{bmatrix}
  \end{bmatrix} \in M_{1,2}(\mathcal{M}^C),
\end{equation*}
which has norm $1$, but
\begin{equation*}
  T^{(1,2)} \Big(
  \begin{bmatrix}
    \begin{bmatrix}
      1 \\ 0
    \end{bmatrix}
    &
    \begin{bmatrix}
      0 \\ 1
    \end{bmatrix}
  \end{bmatrix}
\Big) =
\begin{bmatrix}
  \begin{bmatrix}
    1 & 0
  \end{bmatrix}
  &
  \begin{bmatrix}
    0 & 1
  \end{bmatrix}
\end{bmatrix} \in M_{1,2}(\mathcal{M}^R),
\end{equation*}
which has norm $\sqrt{2}$.

\subsection{Failure of the column-row property in other reproducing kernel Hilbert spaces}
\label{ss:other_RKHS}

While Theorem \ref{thm:main} applies to complete Pick spaces, there are other reproducing kernel Hilbert
spaces that satisfy the column-row poperty. For instance, subspaces of $L^2$-spaces, such as (weighted) Bergman spaces or the Hardy space
on the ball or the polydisc,
trivially satisfy the column-row property with constant $1$.
Nevertheless, there are counterexamples.
We construct a family of rotationally invariant spaces of holomorphic functions on the unit disc that do not satisfy the column-row property with constant $1$.

Let $\alpha > 1$ and for $n \ge 0$, define
\begin{equation*}
  a_n =
  \begin{cases}
    1 & \text{ if } n \neq 2 \\
    \frac{1}{\alpha} & \text{ if } n = 2.
  \end{cases}
\end{equation*}
Let $\mathcal{H}$ be the reproducing kernel Hilbert space on $\mathbb{D}$ with kernel
\begin{equation}
  \label{eqn:kernel}
  K(z,w) = \sum_{n=0}^\infty a_n (z \overline{w})^n.
\end{equation}
It is well known that the monomials $(z^n)_{n=0}^\infty$ form an orthogonal basis of $\mathcal{H}$ with
\begin{equation*}
  \|z^n\|^2 = \frac{1}{a_n} \quad (n \in \mathbb{N});
\end{equation*}
see, for example, \cite[Section 2.1]{CM95}.
Thus, $\mathcal{H} = H^2$, with equivalent but not equal norms.

We will show that the row multiplier norm of the pair $(z,z^2)$ exceeds the column multiplier norm.
Using the fact the monomials form an orthogonal basis, it is easy to check that
\begin{equation*}
  \Big\|
  \begin{bmatrix}
    z \\ z^2
  \end{bmatrix} \Big\|^2_{\Mult}
  = \sup_{n \in \mathbb{N}}
  \frac{1}{\|z^n\|^2}
  \Big\|
  \begin{bmatrix}
    z \\ z^2
  \end{bmatrix} z^n \Big\|^2
  = \sup_{n \in \mathbb{N}} \frac{a_n}{a_{n+1}} + \frac{a_n}{a_{n+2}} = 1 + \alpha,
\end{equation*}
where we used that $\alpha > 1$ in the last step. On the other hand,
\begin{equation*}
  \|
  \begin{bmatrix}
    z & z^2
  \end{bmatrix} \|_{\Mult}^2
  \ge \frac{
    \Big\|
  \begin{bmatrix}
    z & z^2
  \end{bmatrix}
  \begin{bmatrix}
    z \\1
  \end{bmatrix} \Big\|^2
  }{ \Big\|
\begin{bmatrix}
  z \\ 1
\end{bmatrix} \Big\|^2}
= \frac{1}{2} \|2 z^2\|^2 = 2 \alpha.
\end{equation*}
Therefore,
\begin{equation*}
  \frac{
  \|
  \begin{bmatrix}
    z & z^2
  \end{bmatrix} \|_{\Mult}^2
}{
  \Big\|
  \begin{bmatrix}
    z \\ z^2
\end{bmatrix} \Big\|^2_{\Mult}} \ge \frac{2\alpha}{1 + \alpha} > 1.
\end{equation*}
Morever, observe that the last ratio tends to $2$ as $\alpha \to \infty$, which corresponds to the basic $\sqrt{n}$-bound
mentioned in introduction.

It is not hard to check directly that $\mathcal{H}$ is not a complete Pick space.
Indeed, if the reproducing kernel $K$ is of the form \eqref{eqn:kernel}, then $\mathcal{H}$ is a complete Pick space
if and only if the coefficients $(b_n)_{n=1}^\infty$ defined by the power series identity
\begin{equation*}
  \sum_{n=1}^\infty b_n t^n = 1- \frac{1}{\sum_{n=0}^\infty a_n t^n}
\end{equation*}
are all non-negative; see \cite[Theorem 7.33]{AM02}. A small computation shows that
\begin{equation*}
  b_2 = a_2 - a_1^2 = \frac{1}{\alpha} - 1 < 0,
\end{equation*}
so $\mathcal{H}$ is not a complete Pick space.

This example can be easily modified to obtain, for each constant $c > 1$, a reproducing
kernel Hilbert space $\mathcal{H}$ on the unit disc that does not satisfy the column-row property
with constant $c$. Explicitly, one replaces the condition $n=2$ in the definition of $(a_n)$ with the condition
$n=k$ for some large natural number $k$, and compares the column and row norms of the tuple $(z,z^2,\ldots,z^k)$.
Moreover, by taking a direct sum of such examples, one obtains a single reproducing kernel Hilbert space
(on the disjoint union of countably many copies of the unit disc) that admits sequences of multipliers
that induce bounded column, but unbounded row multiplication operators. We omit the details.

\subsection{Questions}

While most of the known applications of the column-row property occur within the realm of complete Pick spaces, one may still ask the following
natural question.

\begin{quest}
  Which reproducing kernel Hilbert spaces satisfy the column-row property (or the column-matrix property)?
\end{quest}

As mentioned in Subsection \ref{ss:other_RKHS}, weighted Bergman spaces and the Hardy space on the ball or the polydisc
satsify the column-row property, but are typically not complete Pick spaces.

For a concrete example, let $\mathcal{H}_a$ be the reproducing kernel Hilbert space on $\mathbb{B}_d$ with kernel
$(1 - \langle z,w \rangle)^{-a}$. Then $\mathcal{H}_a$ is a complete Pick space for $0 < a \le 1$,
and $\Mult(\mathcal{H}_a) = H^\infty(\mathbb{B}_d)$ completely isometrically for $a \ge d$, so $\mathcal{H}_a$
satisfies the column-row property for $a \in (0,1] \cup [d, \infty)$.

\begin{quest}
Does $\mathcal{H}_a$ satisfy the column row property for $a \in (1,d)$?
\end{quest}

Naturally, one can define the column-row property not just in the context of multiplier algebras, but more generally for any operator space.
By slight abuse of terminology, we therefore ask:
\enlargethispage{\baselineskip}

\begin{quest}
  Which operator spaces satisfy the column-row property?
\end{quest}

\bibliographystyle{amsplain}
\bibliography{literature}

\end{document}